\documentclass[]{amsart}   
\usepackage{amsmath}
\usepackage[mathscr]{eucal} 
\usepackage{amssymb}
\usepackage{latexsym}
\usepackage{amsthm} 
\theoremstyle{plain}
\newtheorem{theorem}{Theorem}[section]
\newtheorem{lemma}[theorem]{Lemma}  

\newtheorem{theorema}{Theorem A\!\!}
\newtheorem{theoremb}{Theorem B\!\!}
\newtheorem{theoremc}{Theorem C\!\!}
\newtheorem{theoremd}{Theorem D\!\!}
\newtheorem{theoreme}{Theorem E\!\!}
\newtheorem{theoremf}{Theorem F\!\!}

\newtheorem{proposition}[theorem]{Proposition}

\newtheorem{corollary}[theorem]{Corollary}

\newcommand{\supp}{\mathop{\mathrm{supp}}\nolimits} 
\newcommand{\sgn}{\mathop{\mathrm{sgn}}\nolimits}

\theoremstyle{definition} 

\newtheorem{definition}[theorem]{Definition}
\numberwithin{equation}{section} 
\newtheorem{remark}[theorem]{Remark}

\theoremstyle{remark}

\def\Xint#1{\mathchoice
{\XXint\displaystyle\textstyle{#1}}%
{\XXint\textstyle\scriptstyle{#1}}%
{\XXint\scriptstyle\scriptscriptstyle{#1}}%
{\XXint\scriptscriptstyle\scriptscriptstyle{#1}}%
\!\int}
\def\XXint#1#2#3{{\setbox0=\hbox{$#1{#2#3}{\int}$}
\vcenter{\hbox{$#2#3$}}\kern-.5\wd0}}

\def\dashint{\Xint-}

\title[Littlewood-Paley equivalence  ]
{Littlewood-Paley equivalence and homogeneous Fourier multipliers}  
\author{Shuichi Sato} 
 
\begin{document} 

\address{Department of Mathematics,
Faculty of Education, Kanazawa University, Kanazawa 920-1192, Japan}
\email{shuichi@kenroku.kanazawa-u.ac.jp}
\begin{abstract} 
We consider certain Littlewood-Paley operators and prove 
characterization of some function spaces in terms of those operators.  
When treating weighted 
Lebesgue spaces, a generalization to weighted spaces  will be made for 
 H\"{o}rmander's theorem  on the invertibility of homogeneous Fourier 
 multipliers.  
Also, applications to the theory of Sobolev spaces will be given.   
\end{abstract}
  \thanks{ 
  \endgraf
    }
\thanks{The author is partly supported
by Grant-in-Aid for Scientific Research (C) No. 25400130, Japan Society for the  Promotion of Science.} 
\maketitle 
\section{Introduction}  
Let $\psi$ be a function in  $L^1(\Bbb R^n)$  such that  
\begin{equation}\label{cancell}
\int_{\Bbb R^n} \psi (x)\,dx = 0.  
\end{equation} 
We consider 
the Littlewood-Paley function on $\Bbb R^n$ defined by 
\begin{equation}\label{lpop}
g_{\psi}(f)(x) = \left( \int_0^{\infty}|f*\psi_t(x)|^2
\,\frac{dt}{t} \right)^{1/2},   
\end{equation} 
where $\psi_t(x)=t^{-n}\psi(t^{-1}x)$.  
The following result of Benedek, 
Calder\'on and Panzone \cite{BCP} on the $L^p$ boundedness, $1<p<\infty$, of 
$g_\psi$ is well-known. 
\begin{theorema}
We assume \eqref{cancell} for $\psi$ and   
\begin{gather}\label{1.2} 
|\psi(x)|\leq C(1+|x|)^{-n-\epsilon}, 
 \\  \label{1.3}
\int_{\Bbb R^n}|\psi(x-y) - \psi(x)|\,dx \leq C|y|^\epsilon  
\end{gather} 
for some positive constant $\epsilon$.   
Then  $g_\psi$  is bounded on $L^p(\Bbb R^n)$ for all 
$p\in (1, \infty):$  
\begin{equation}\label{bcp}
  \|g_\psi(f)\|_p\leq C_p\|f\|_p,    
\end{equation}
where 
$$\|f\|_p=\|f\|_{L^p}=\left(\int_{\Bbb R^n}|f(x)|^p\, dx\right)^{1/p}. $$
\end{theorema}  
\par 
By the Plancherel theorem, it follows that $g_\psi$ is bounded on 
$L^2(\Bbb R^n)$ if and only if $m\in L^\infty(\Bbb R^n)$, where 
$m(\xi)=\int_0^\infty |\hat{\psi}(t\xi)|^2\, dt/t$, 
which is a homogeneous function of degree $0$.   Here  
the Fourier transform is defined as 
\begin{equation*}
\hat{\psi}(\xi)=\int_{\Bbb R^n} \psi(x) 
e^{-2\pi i \langle x,\xi\rangle}\, dx, \quad  
\langle x,\xi\rangle=\sum_{k=1}^n x_k\xi_k.    
\end{equation*} 
\par  
Let  
$$P_t(x)=c_n \frac{t}{(|x|^2+t^2)^{(n+1)/2}}$$ 
be the Poisson kernel on the upper half space $\Bbb R^n \times (0, \infty)$ 
and $Q(x)= [(\partial/\partial t) P_t(x)]_{t=1}$.   
Then, we can see that the function $Q$ satisfies the conditions 
\eqref{cancell}, \eqref{1.2} and \eqref{1.3}.  Thus 
 by Theorem A $g_{Q}$ is bounded on $L^p(\Bbb R^n)$ 
for all $p\in (1, \infty)$. 
\par 
Let 
$H(x)= \sgn(x)\chi_{[-1,1]}(x)=\chi_{[0,1]}(x)-\chi_{[-1,0]}(x)$  on $\Bbb R$ 
(the Haar function), 
where $\chi_E$ denotes the characteristic function of a set $E$ and 
$\sgn(x)$ the signum function.    
Then $g_{H}(f)$ is the Marcinkiewicz integral 
$$ \mu(f)(x)= \left( \int_0^{\infty}|F(x+t)+F(x-t)-2F(x)|^2\,\frac{dt}{t^3} \right)^{1/2},$$  
where $F(x)=\int_0^xf(y)\,dy$.   
Also, we can easily see that Theorem A implies that 
$g_H$ is bounded on $L^p(\Bbb R)$, $1<p<\infty$. 
\par 
Further, we can consider 
  the generalized Marcinkiewicz integral $\mu_{\alpha}(f)$ ($\alpha>0$)  on 
  $\Bbb R$ defined by 
 $$\mu_{\alpha}(f)(x) = \left( \int_0^{\infty}|S_t^{\alpha}(f)(x)|^2
 \,\frac{dt}{t}\right)^{1/2},$$ 
 where 
\begin{equation*}\label{gemar}
S_t^{\alpha}(f)(x) = \frac{\alpha}{t} \int_0^t\left(1 - \frac{u}{t}\right)^{\alpha - 1}
\left(f(x - u) - f(x + u)\right)\,du.  
\end{equation*}  
We observe that $\mu_\alpha(f)= g_{\varphi^{(\alpha)}}(f)$  
with 
\begin{equation}\label{1.4}
 \varphi^{(\alpha)}(x)= 
 \alpha|1-|x||^{\alpha -1}\sgn(x)\chi_{(-1, 1)}(x).    
\end{equation}  
  The square function $\mu_1$ coincides with the ordinary Marcinkiewicz integral $\mu$.  
When $\psi$ is compactly supported, relevant sharp results for the $L^p$ boundedness 
of $g_\psi$ can be found in \cite{D, FS, Sa3}.  
\par 
We can also consider Littlewood-Paley operators on the Hardy space 
 $H^p(\Bbb R^n)$, $0<p<\infty$. We consider 
a dense subspace $\mathscr S_0(\Bbb R^n)$ of $H^p(\Bbb R^n)$ consisting 
of those functions $f$ in $\mathscr S(\Bbb R^n)$ which satisfy $\hat{f}=0$  
near the origin, where $\mathscr S(\Bbb R^n)$ denotes the Schwartz class of 
rapidly decreasing smooth functions. 
  Let $f \in  \mathscr S_0(\Bbb R)$. 
Then, if $2/(2\alpha + 1) < p < \infty$ and $\alpha>0$,  we have 
$\|\mu_{\alpha}(f)\|_p \simeq \|f\|_{H^p}$, which means 
 \begin{equation}\label{1.5} 
 c_p\|f\|_{H^p}\leq \|\mu_{\alpha}(f)\|_p \leq C_p\|f\|_{H^p}  
\end{equation} 
with some positive constants $c_p, C_p$ independent of $f$ 
(see \cite{Su}, \cite{Sa2}).  
\par 
To state results about the reverse inequality of \eqref{bcp}, we first recall 
a theorem of H\"{o}rmander \cite{H}.  
Let $m\in L^\infty(\Bbb R^n)$ and define 
\begin{equation}\label{fmo}
T_m(f)(x)=\int_{\Bbb R^n} m(\xi)\hat{f}(\xi)
e^{2\pi i \langle x,\xi\rangle}\, d\xi.    
\end{equation} 
We say that $m$ is a Fourier multiplier for $L^p$ and write $m\in M^p$ 
if there exists a constant $C>0$ such that 
\begin{equation*}\label{dfm}
\|T_m(f)\|_{p}\leq C\|f\|_{p}
\end{equation*} 
for all $f\in L^2\cap L^p$.  
Then the result of H\"{o}rmander \cite{H} can be stated as follows.
\begin{theoremb} 
Let $m$ be a bounded function on $\Bbb R^n$ which is homogeneous of 
degree $0$. 
Suppose that  $m\in M^p$ for all $p\in (1,\infty)$. Suppose further 
that $m$ is continuous and does not vanish on 
$S^{n-1}=\{x\in \Bbb R^n: |x|=1\}$.   
Then, $m^{-1}\in M^p$ for every $p\in (1,\infty)$. 
\end{theoremb} 
See \cite{CZ, BCP} for related results. 
Applying Theorem B, we can deduce the following (see \cite[Theorem 3.8]{H}). 
\begin{theoremc} 
Suppose that $g_\psi$ is bounded on $L^p$ for every $p\in 
(1,\infty)$.  Let $m(\xi)=\int_0^\infty |\hat{\psi}(t\xi)|^2\, dt/t$. 
If $m$ is continuous and strictly positive on $S^{n-1}$, then we have 
$$ \|f\|_p\leq c_p \|g_\psi(f)\|_p,   $$  
and hence $\|f\|_p \simeq \|g_\psi(f)\|_p$, $f\in L^p$, 
for all $p\in (1, \infty)$.   
\end{theoremc}
\par 
In this note we shall generalize Theorems B and C to weighted 
$L^p$ spaces with $A_p$ weights of Muckenhoupt 
(see Theorems 2.5, 2.9 and Corollaries 2.6,  2.11). Our proof 
of Theorem 2.5 has some features in common with the proof of Wiener-L\'{e}vy 
theorem in \cite[vol. I, Chap. VI]{Z}.   
We also consider a discrete parameter version of $g_\psi$: 
\begin{equation}\label{dlpop}
\Delta_\psi(f)(x)
=\left(\sum_{k=-\infty}^{\infty}\left|f*\psi_{2^k}(x)\right|^2  \right)^{1/2}. 
\end{equation} 
We shall have $\Delta_\psi$ analogues of results for $g_\psi$ 
(see Theorem 3.5 and Corollary 3.7).  
We formulate Theorems 2.9 and 3.5 in general forms so that they include 
unweighted cases as special cases, while Corollaries 2.11 and 3.7 may be 
more convenient for some applications. 
\par 
In the unweighted case, we shall prove some results on $H^p$ analogous to 
Corollaries 2.11 and 3.7 for $p$ close to $1$, $p\leq 1$, in Section 4 under 
a certain regularity condition for $\psi$ (Theorems 4.7 and 4.8).   
We shall consider functions $\psi$ including those which cannot be treated 
directly by the theory of \cite{U}.  As a result, 
in particular, we shall be able to give a proof of the second inequality of 
\eqref{1.5} for $1/2<\alpha<3/2$ and $2/(2\alpha + 1) < p \leq 1$ 
by methods of real analysis which 
does not depend on the Poisson kernel. 
\par 
Here we recall some more background materials on $\mu_\alpha$.  
When $p<1$ and $1/2<\alpha<1$, we know proofs for the first 
and the second inequality of 
\eqref{1.5} which use  pointwise relations 
$\mu_\alpha(f)\geq c g_0(f)$ and
 $\mu_\alpha(f)\sim g^*_{\lambda}(f)$ with $\lambda=1+2\alpha$, respectively, 
 and apply appropriate properties of $g_0$ and $g^*_{\lambda}$.   
Also, we note that a proof of the inequality 
 $\|\mu(f)\|_1 \leq C\|f\|_{H^1}$ using a theory of vector valued singular 
integrals can be found in  \cite[Chap. V]{GR} (see also \cite{RRT}).        
We have assumed that $\supp(\hat{f})\subset [0,\infty)$ in stating 
$\mu_\alpha(f)\sim g^*_{\lambda}(f)$ and 
$g_0(f)$,  $g^*_{\lambda}(f)$  are the Littlewood-Paley functions 
defined by 
\begin{equation*}
g_0(f)(x)= \left( \int_0^{\infty}
|(\partial/\partial x) u(x,t)|^2t \, dt \right)^{1/2},  
\end{equation*}
\begin{equation*}
g^*_{\lambda}(f)(x) = \left(\iint_{\Bbb R \times (0,\infty) }
\left(\frac{t}{t+|x-y|}
\right)^{\lambda}|\nabla u(y,t)|^2 \,dy\,dt\right)^{1/2}  
\end{equation*}
with $u(y,t)$ denoting the Poisson integral of $f$: $u(y,t) = P_t *f(y)$ 
(see \cite{Su}, \cite{Sa2} and references therein,  
and also \cite{KS},  \cite{MW} for related results). 
\par 
 In \cite{U}, a  proof of $\|f\|_{H^p}\leq C\|g_Q(f)\|_p$ 
on $\Bbb R^n$ is given without the use of harmonicity
(see \cite{FeS}  for the original proof using properties of harmonic 
functions). Also, when $n=1$, a similar result is shown for $g_0$. 
 It is to be noted that, combining this with the pointwise relation 
 between  $g_0$ and $\mu_\alpha$ mentioned above, 
we can  give a proof of  the first inequality of  \eqref{1.5} for the whole 
range of $p$, $\alpha$  in such a manner that  a special property of the 
Poisson  kernel is  used only to prove the pointwise relation.   
\par 
In Section 5, we shall apply Corollaries 2.11 and 3.7 to the theory of  
Sobolev spaces.   
In \cite{AMV}, the operator  
\begin{equation}\label{cb1}  
U_\alpha(f)(x)=\left(\int_0^\infty\left|f(x)-\dashint_{B(x,t)}f(y)\, dy
\right|^2\frac{dt}{t^{1+2\alpha}} \right)^{1/2}, \quad \alpha>0, 
\end{equation}  
was studied, where $\dashint_{B(x,t)} f(y)\,dy$ is defined as     
$|B(x,t)|^{-1}\int_{B(x,t)} f(y)\,dy$ with $|B(x,t)|$ denoting 
the Lebesgue measure of a ball $B(x,t)$ in $\Bbb R^n$ of radius $t$ 
centered at $x$.  
The operator $U_1$ was used to characterize the Sobolev space 
$W^{1,p}(\Bbb R^n)$.    
\begin{theoremd} Let $1<p<\infty$. Then, the following two statements are 
equivalent$:$     
\begin{enumerate} \renewcommand{\labelenumi}{(\arabic{enumi})}  
\item $f$ belongs to $W^{1,p}(\Bbb R^n)$, 
\item  $f\in L^p(\Bbb R^n)$ and $U_1(f)\in L^p(\Bbb R^n)$.  
\end{enumerate}  
Furthermore, from either of the two conditions $(1), (2)$ it follows that 
$$\|U_1(f)\|_p\simeq \|\nabla f\|_p.   $$
\end{theoremd} 
This may be used to define a Sobolev space analogous to $W^{1,p}(\Bbb R^n)$ 
in metric measure spaces. 
We shall also consider a discrete parameter version of $U_\alpha$:   
\begin{equation}\label{db1} 
E_\alpha(f)(x)=\left(\sum_{k=-\infty}^\infty\left|f(x)-\dashint_{B(x,2^k)}f(y)
\, dy\right|^2 2^{-2k\alpha} \right)^{1/2}, \quad \alpha>0,   
\end{equation} 
and prove an analogue of Theorem D for $E_\alpha$.   
Further, we shall consider operators generalizing $U_\alpha$, $E_\alpha$ 
and show that they can be used to characterize the weighted Sobolev spaces, 
focusing on the case $0<\alpha<n$.

\section{Invertibility of homogeneous Fourier multipliers and Littlewood-Paley 
operators} 
We say that a weight function $w$ belongs to the weight class $A_p$, 
 $1<p< \infty$, of Muckenhoupt on $\Bbb R^n$ if 
 $$[w]_{A_p}=  
 \sup_B \left(|B|^{-1} \int_B w(x)\,dx\right)\left(|B|^{-1} \int_B
w(x)^{-1/(p-1)}dx\right)^{p-1} < \infty, $$
where the supremum is taken over all balls $B$ in $\Bbb R^n$. 
Also,  we say that $w\in A_1$ if $M(w)\leq Cw$ almost everywhere, 
with $M$ denoting the Hardy-Littlewood maximal operator 
$$M(f)(x)=\sup_{x\in B}|B|^{-1}\int_B|f(y)|\,dy,   $$ 
where the supremum is taken over all balls $B$ in $\Bbb R^n$ containing  $x$;  
we denote by $[w]_{A_1}$ the infimum of all such $C$. 
\par 
Let $m\in L^\infty(\Bbb R^n)$ and $w\in A_p$, $1<p<\infty$.  
Let $T_m$ be as in (1.8). 
We say that $m$ is a Fourier multiplier for $L^p_w$ and write
 $m\in M^p(w)$ if there exists a constant $C>0$ such that 
\begin{equation}\label{1}
\|T_m(f)\|_{p, w}\leq C\|f\|_{p,w}
\end{equation} 
for all $f\in L^2\cap L^p_w$, where 
$$\|f\|_{p,w}=\|f\|_{L^p_w}= \left(\int_{\Bbb R^n}|f(x)|^pw(x)
\, dx\right)^{1/p}.   $$   
We also write $L^p(w)$ for $L^p_w$.  
Define 
$$\|m\|_{M^p(w)}=\inf C,  $$ 
where the infimum is taken over all the constants $C$ satisfying \eqref{1}.   
Since $L^2\cap L^p_w$ is dense in $L^p_w$,   $T_m$ 
uniquely extends to a bounded linear operator on $L^p_w$ if $m\in M^p(w)$.  
In this note we shall confine our attention to the case of $L^p_w$ boundedness 
of $T_m$ with $w\in A_p$.    
We note that $M^p(w)=M^{p'}(\widetilde{w}^{-p'/p})$ by duality, where 
$1/p+1/p'=1$ and $\widetilde{w}(x)=w(-x)$. 
\par 
If $w\in A_p$, $1<p<\infty$,  then $w^s\in A_r$ for some $s>1$ and $r<p$ 
(see \cite{GR}). In applying interpolation arguments it is useful  
if sets of those $(r,s)$ are specially notated.  
\begin{definition} Let $w\in A_p$, $1<p<\infty$.   
For $0<\sigma<p-1$, 
$\tau>0$, set 
$$U(w,p)=U(w,p,\sigma, \tau)=(p-\sigma,p+\sigma)\times [1, 1+\tau). $$ 
We say that $U(w,p)$ is a $(w,p)$  set  
if $w^s \in A_r$ for all $(r,s)\in U(w,p)$. 
We  write 
$m\in M(U(w,p))$ if  $m\in M^r(w^s)$ for all $(r,s)\in U(w,p)$.   
\end{definition}  
\par 
We need a relation of $\|m\|_{M^p(w)}$ and $\|m\|_\infty$ in the following. 
\begin{proposition} 
Let $w\in A_p$, $1<p<\infty$. Suppose that $m\in M(U(w,p))$  for 
a $(w,p)$ set $U(w,p)$.  
Then 
\begin{align*} 
\|m\|_{M^p(w)}&\leq \|m\|_\infty^{1-\theta}\|m\|^\theta_{M^{p+\delta}
(w^{1+\epsilon})}, \quad (p+\delta,1+\epsilon)\in U(w,p),    
\quad \text{if $p>2$}; 
\\  
\|m\|_{M^p(w)}&\leq \|m\|_\infty^{1-\theta}\|m\|^\theta_{M^{p}
(w^{1+\epsilon})}, \quad (p,1+\epsilon)\in U(w,p),    
\quad \text{if $p=2$}; 
\\ 
\|m\|_{M^p(w)}&\leq \|m\|_\infty^{1-\theta}\|m\|^\theta_{M^{p-\delta}
(w^{1+\epsilon})}, \quad (p-\delta,1+\epsilon)\in U(w,p), 
 \quad \text{if $1<p<2$}, 
\end{align*} 
for some $\theta\in(0,1)$ and some small numbers $\delta, \epsilon>0$.  
\end{proposition}   
\begin{proof} 
Let $1<p<2$ and $w\in A_p$.  Then, there exist $\epsilon_0, \delta_0>0$ such 
that $(p-\delta,1+\epsilon)\in U(w,p)$ for all 
$\epsilon\in (0,\epsilon_0]$ and $\delta\in (0,\delta_0]$.     
Let $1/p=(1-\theta)/2+\theta/(p-\delta)$, $\delta\in (0, \delta_0]$. Then, 
since $m\in M^{p-\delta}(w^{1+\epsilon})$, by interpolation with change of 
measures of Stein-Weiss (see \cite{BL}) 
between $L^2$ and $L^{p-\delta}(w^{1+\epsilon})$ 
boundedness, we have 
$$\|m\|_{M^p(w^{p\theta(1+\epsilon)/(p-\delta)})}\leq \|m\|_{\infty}^{1-\theta}
\|m\|_{M^{p-\delta}(w^{1+\epsilon})}^{\theta}. $$ 
Note that 
$p\theta/(p-\delta)= (2-p)/(2-p+\delta)$.    Thus we can choose $\epsilon, 
\delta>0$ so that $p\theta(1+\epsilon)/(p-\delta)=1$.   This completes the proof for $p\in (1,2)$.   
\par 
Suppose that $2<p<\infty$ and $w\in A_p$. Then there exist $\epsilon_0>0$ 
$\delta_0>0$ such that $(p+\delta,1+\epsilon)\in U(w,p)$ for all 
$\epsilon\in (0,\epsilon_0]$ and $\delta\in (0,\delta_0]$.  So, 
$m\in M^{p+\delta}(w^{1+\epsilon})$ for all $\epsilon\in (0,\epsilon_0]$ and 
$\delta\in (0,\delta_0]$.  
Similarly to the case $1<p<2$, applying interpolation, we have 
$$\|m\|_{M^p(w^{p\theta(1+\epsilon)/(p+\delta)})}\leq \|m\|_{\infty}^{1-\theta}
\|m\|_{M^{p+\delta}(w^{1+\epsilon})}^{\theta} $$ 
with $p\theta/(p+\delta)= (p-2)/(p+\delta-2)$.
 Taking $\epsilon, \delta$ so that $p\theta(1+\epsilon)/(p+\delta)=1$, 
we conclude the proof for $p\in (2, \infty)$. 
The case $p=2$ can be handled similarly. 
\end{proof}
\par 
To treat Fourier multipliers arising from Littlewood-Paley functions in 
\eqref{lpop}  and\eqref{dlpop} simultaneously, we slightly generalize the 
usual notion of homogeneity.   
\begin{definition} 
Let $f$ be a function on $\Bbb R^n$. We say that $f$ is dyadically homogeneous 
of degree $\tau$, $\tau \in \Bbb R$,  if $f(2^k x)=2^{k\tau}f(x)$ for all  
$x\in \Bbb R^n\setminus\{0\}$ and all $k\in \Bbb Z$ $($the set of integers$)$. 
\end{definition} 
\par 
For $m\in M^p_w$, $1<p<\infty$, $w\in A_p$, we consider the spectral radius 
operator  
\begin{equation*}
\rho_{p,w}(m)=\lim_{k\to \infty}\|m^k\|_{M^p(w)}^{1/k}.  
\end{equation*}  
To prove a weighted version of Theorem B, we need an approximation result for 
Fourier multipliers in $M^p(w)$.  
\begin{proposition}    Let $1<p<\infty$,  $w\in A_p$ and 
$m\in L^\infty(\Bbb R^n)$.  
We assume that $m$ is dyadically homogeneous of degree $0$ and 
continuous on the closed annulus $B_0=\{\xi\in \Bbb R^n: 1\leq |\xi|\leq 2\}$. 
We further assume that there exists a $(w,p)$ set $U(w,p)$ such that  
$m\in M(U(w,p))$.  
Then, for any $\epsilon>0$, there exists $n\in M^p(w)$ 
which is  dyadically homogeneous of degree $0$ and in 
$C^\infty(\Bbb R^n\setminus\{0\})$ such that $\|m-n\|_{\infty}<\epsilon$ 
and $\rho_{p,w}(m-n)<\epsilon$.  
\end{proposition} 
\begin{proof}
Let $\{\varphi_j\}_{j=1}^\infty$ be a sequence of functions on $O(n)$ 
such that 
\begin{enumerate} 
\item[$\bullet$] each $\varphi_j$ is infinitely differentiable and 
non-negative, 
\item[$\bullet$] for any neighborhood $U$ of the identity in $O(n)$,  there 
exists a positive integer $N$ such that $\supp(\varphi_j)\subset U$ 
if $j\geq N$,   
\item[$\bullet$] $\int_{O(n)} \varphi_j(A)\, dA= 1$, where $dA$ is the Haar 
measure on $O(n)$.  
\end{enumerate} 
Also, let $\{\psi_j\}_{j=1}^\infty$ be a sequence of non-negative functions 
in  $C^\infty(\Bbb R)$ such that $\supp(\psi_j)\subset [1-2^{-j}, 1+2^{-j}]$ 
and $\int_0^\infty \psi_j(t)\, dt/t =1$.  
Define 
$$m_j(\xi)=\int_0^\infty \int_{O(n)} m(tA\xi)\varphi_j(A)\psi_j(t)\, dA\, 
\frac{dt}{t}. $$  
Then $m_j$ is dyadically homogeneous of degree $0$, infinitely differentiable 
and $m_j\to m$ uniformly in $\Bbb R^n\setminus\{0\}$ 
by the continuity of $m$ on $B_0$. This can be shown similarly to 
\cite[pp. 123-124]{H}, where we can find the case when $m$ is homogeneous of 
degree $0$.    
Also, for a positive integer $k$, the derivatives of $m_j^k$ satisfy 
\begin{equation}\label{deriv}
|\partial_\xi^\gamma m_j(\xi)^k|\leq C_{j,k,M}\|m\|_\infty^k 
 |\xi|^{-|\gamma|}, \quad 
\partial_\xi^\gamma
=(\partial/\partial \xi_1)^{\gamma_1}\dots 
(\partial/\partial \xi_n)^{\gamma_n} 
\end{equation}  
for all multi-indices  $\gamma$ with $|\gamma| \leq M$, where  $M$ is 
any positive integer,  
$\gamma=(\gamma_1, \dots, \gamma_n)$, $|\gamma|=\gamma_1+\dots+\gamma_n$, 
$\gamma_j\in \Bbb Z$, $\gamma_j\geq 0$ and we have  
$C_{j,k,M}\leq C_{j, M}k^M$ with a constant $C_{j, M}$ independent of $k$. 
By \eqref{deriv}, if $M$ is sufficiently large,  it follows that 
\begin{equation*}
\|m_j^k\|_{M^p(w)}\leq CC_{j,k,M}\|m\|_\infty^k, \quad w\in A_p, 1<p<\infty 
\end{equation*} 
with a constant $C$ independent of $k$ (see \cite{CF, KW}). Thus, by the 
evaluation of  $C_{j,k,M}$ we have 
\begin{equation}\label{est1}
\rho_{p, w}(m_j)\leq \|m\|_\infty.   
\end{equation} 
\par 
Since $m, m_j\in M(U(w,p))$, 
by Proposition $2.2$, we can find $r$ close to $p$, $s>1$ with 
$(r,s)\in U(w,p)$  and $\theta\in (0,1)$ such that 
$$\|(m-m_j)^k\|_{M^p(w)}\leq \|(m-m_j)^k\|_\infty^{1-\theta}
\|(m-m_j)^k\|_{M^r(w^{s})}^\theta. $$  
It follows  that 
\begin{equation*}
\rho_{p,w}(m-m_j)\leq \|m-m_j\|_\infty^{1-\theta}\rho_{r, w^s}(m-m_j)^\theta. 
\end{equation*} 
Thus, by \eqref{est1} we have 
\begin{align*}\label{est2}
\rho_{p,w}(m-m_j)
&\leq \|m-m_j\|_\infty^{1-\theta}(\rho_{r, w^s}(m)^\theta+
\rho_{r, w^s}(m_j)^\theta) 
\\ 
&\leq \|m-m_j\|_\infty^{1-\theta}(\rho_{r, w^s}(m)^\theta+
\|m\|_\infty^\theta). 
\end{align*} 
 This completes the proof since 
 $\|m-m_j\|_\infty \to 0$ as $j\to \infty$. 
\end{proof}
\par 
Applying Proposition $2.4$, we can generalize Theorem B as follows.   
\begin{theorem}  Suppose that $1<p<\infty,  w\in A_p$ and that $m\in 
L^\infty(\Bbb R^n)$ fulfills the 
hypotheses of Proposition $2.4$. Also, suppose that $m(\xi)\neq 0$ for every 
$\xi\neq 0$. 
Let $\varphi(z)$ be holomorphic in $\Bbb C\setminus\{0\}$.  Then we have 
$\varphi(m(\xi))\in M^p(w)$. 
\end{theorem}
\begin{proof} 
Define $\epsilon_0>0$ by  
$$4\epsilon_0=\min_{\xi\in \Bbb R^n\setminus\{0\}}|m(\xi)| 
=\min_{1\leq|\xi|\leq 2}|m(\xi)|. 
$$ 
By Proposition $2.4$, there is $n\in M^p(w)$ which is dyadically 
homogeneous of 
degree $0$ and infinitely differentiable in $\Bbb R^n\setminus\{0\}$ such that 
$\|m-n\|_{\infty}<\epsilon_0$ and $\rho_{p,w}(m-n)<\epsilon_0$.  
If we consider a curve $C: n(\xi)+2\epsilon_0e^{i\theta}, 0\leq \theta\leq 
2\pi$, 
Cauchy's formula  can be applied to  
represent $\varphi(m(\xi))$ by a contour integral as follows:
$$\varphi(m(\xi))= \frac{1}{2\pi i}\int_C\frac{\varphi(\zeta)}{\zeta-m(\xi)}\, 
d\zeta
=\frac{\epsilon_0}{\pi}\int_0^{2\pi}
\frac{\varphi(n(\xi)+2\epsilon_0e^{i\theta})}{2\epsilon_0e^{i\theta}+n(\xi)-
m(\xi)}e^{i\theta}\, d\theta, \quad \xi\neq 0.  $$   
Note that 
$$\frac{e^{i\theta}}{2\epsilon_0e^{i\theta}+n(\xi)-m(\xi)}
=\frac{1}{2\epsilon_0}\sum_{k=0}^\infty
\left(\frac{m(\xi)-n(\xi)}{2\epsilon_0e^{i\theta}}\right)^k; $$  
the series converges uniformly in $\theta\in [0,2\pi]$ since $|m(\xi)-n(\xi)|
<\epsilon_0$.    Thus 
$$\varphi(m(\xi))=\frac{1}{2\pi}\sum_{k=0}^\infty 
\left(\frac{m(\xi)-n(\xi)}{2\epsilon_0}\right)^k M_k(\xi)
$$ 
uniformly in $\Bbb R^n\setminus\{0\}$, where 
$$M_k(\xi)
 =\int_0^{2\pi}\varphi(n(\xi)+2\epsilon_0e^{i\theta})e^{-ik\theta}\, d\theta.
$$ 
\par 
Since $|n(\xi)+2\epsilon_0e^{i\theta}|\geq \epsilon_0$, we can see that 
$M_k(\xi)$ is dyadically homogeneous of degree $0$ and infinitely 
differentiable in $\Bbb R^n\setminus \{0\}$; also the derivative satisfies  
$$|\partial_\xi^\gamma M_k(\xi)|\leq C_\gamma |\xi|^{-|\gamma|} 
$$  
for every multi-index $\gamma$ with a constant $C_\gamma$ independent of 
$k$. 
This implies that $\|M_k\|_{M^p(w)}
\leq C$ with a constant $C$ independent of $k$ (see \cite{CF, KW}).  
Thus we have $\varphi(m)\in M^p(w)$ and 
$$\|\varphi(m)\|_{M^p(w)}\leq \frac{1}{2\pi}\sum_{k=0}^\infty 
(2\epsilon_0)^{-k}\|(m-n)^k\|_{M^p(w)}\|M_k\|_{M^p(w)},   $$   
since the series converges, for $\|(m-n)^k\|_{M^p(w)}\leq \epsilon_0^k$ 
if $k$ is sufficiently large.  
This completes the proof.   

\end{proof}  
 Theorem 2.5 in particular implies the following. 
\begin{corollary}\label{inverse}  Let  $1<p<\infty$ and $w\in A_p$. 
Let $m$ be a dyadically homogeneous function of degree $0$ such that 
$m\in M^r(v)$ for all 
$r\in (1,\infty)$ and all $v\in A_r$.  
We  assume that $m$ is continuous on $B_0$ and does not 
vanish there. Then $m^{-1}\in M^p(w)$.   
\end{corollary}
\par 
We have  applications of Theorem 2.5 and  Corollary \ref{inverse} 
 to the theory of Littlewood-Paley operators. 
 Let $w\in A_p$, $1<p<\infty$. We say that $g_\psi$ of \eqref{lpop} 
 is bounded on $L^p_w$ if 
there exists a constant $C$ such that $\|g_\psi(f)\|_{p,w}\leq C\|f\|_{p,w}$ 
for $f\in L^2_w\cap L^2$.  The unique sublinear extension on $L^p_w$ is also 
denoted by $g_\psi$.  The $L^p_w$ boundedness for $\Delta_\psi$ of 
\eqref{dlpop}  is considered similarly. 
\par 
Let $\mathscr H$ be the Hilbert space of 
functions $u(t)$ on $(0,\infty)$ such that $\|u\|_{\mathscr H}=
\left(\int_0^\infty|u(t)|^2\, dt/t\right)^{1/2}<\infty$.  
We consider weighted spaces 
$L^p_{w,\mathscr H}$ of functions 
$h(y,t)$ with the norm  
$$\|h\|_{p,w, \mathscr H}=\left(\int_{\Bbb R^n}\|h^y\|_{\mathscr H}^p w(y)\, dy
\right)^{1/p}, $$ 
where $h^y(t)=h(y,t)$. If $w=1$ identically, the spaces $L^p_{w,\mathscr H}$ 
will be written simply as $L^p_{\mathscr H}$.  
\par 
Define 
\begin{equation}
E_\psi^\epsilon(h)(x)=\int_0^\infty\int_{\Bbb R^n} 
\psi_t(x-y)h_{(\epsilon)}(y,t)\,dy\,\frac{dt}{t}, 
\end{equation} 
where  $h\in L^2_{\mathscr H}$ and $h_{(\epsilon)}(y,t)=h(y,t)\chi_{(\epsilon, \epsilon^{-1})}(t)$, $0<\epsilon<1$,  and 
we assume that $\psi\in L^1(\Bbb R^n)$ with \eqref{cancell}.  
\par 
Then we have the following.  

\begin{lemma} Let $1<r<\infty$ and $v\in A_r$.  We assume that 
$$\|g_\psi(f)\|_{r',v^{-r'/r}}\leq C_0(r,v)\|f\|_{r',v^{-r'/r}}.   $$ 
Then, if  $h \in L^r_{v,\mathscr H}\cap L^2_{\mathscr H}$, 
we have 
$$\sup_{\epsilon\in (0,1)}\|E_{\widetilde{\bar{\psi}}}^\epsilon 
(h)\|_{r,v}  
\leq C_0(r,v)\|h\|_{r,v, \mathscr H},    
$$  
where $\bar{\psi}$ denotes the complex conjugate.  
\end{lemma}  
\begin{proof}  
For $f\in \mathscr S(\Bbb R^n)$, we see that 
\begin{align*}
\left|\int_{\Bbb R^n} E_{\widetilde{\bar{\psi}}}^\epsilon(h)(x)f(x)
\, dx\right| &=\left|\int_{\Bbb R^n} 
\left(\int_\epsilon^{\epsilon^{-1}}\int_{\Bbb R^n} 
\widetilde{\bar{\psi}}_t(x-y)h(y,t)\,dy
\,\frac{dt}{t}\right)f(x)\, dx\right| 
\\ 
&=\left|\int_\epsilon^{\epsilon^{-1}}\int_{\Bbb R^n} \bar{\psi}_t*f(y)
h(y,t)\,dy\,\frac{dt}{t}\right|
\\ 
&\leq \int_{\Bbb R^n} g_{\psi}(\bar{f})(y) \|h^y\|_{\mathscr H}\, dy. 
\end{align*}
Thus, by H\"{o}lder's inequality, we have  
\begin{align*} 
\left|\int_{\Bbb R^n} E_{\widetilde{\bar{\psi}}}^\epsilon(h)(x)f(x)
\, dx\right| 
&\leq \|g_{\psi}(\bar{f})\|_{r', v^{-r'/r}}
\left(\int\|h^y\|_{\mathscr H}^r v(y)\, dy\right)^{1/r}
\\ 
&\leq C_0(r,v)\|f\|_{r', v^{-r'/r}}
\left(\int\|h^y\|_{\mathscr H}^r v(y)\, dy\right)^{1/r}.  
\end{align*}
Taking the supremum over $f$ with $\|f\|_{r', v^{-r'/r}}\leq 1$, we get the 
desired result.  
\end{proof}  

 By applying Lemma 2.7, we have the following. 
\begin{proposition}  
Suppose that $g_\psi$ satisfies   
the hypothesis of Lemma $2.7$ with $r\in (1,\infty)$ and $v\in A_r$. Also, 
we assume that    
$$\|g_\psi(f)\|_{r,v}\leq C_1(r,v)\|f\|_{r,v}.   $$ 
Put 
$$ m(\xi)= \int_0^\infty |\hat{\psi}(t\xi)|^2\, \frac{dt}{t}. $$  
Then 
$\|m\|_{M^r(v)}\leq C_0(r,v)C_1(r,v)$.   
\end{proposition} 
\begin{proof}  
We first note that an interpolation with change of measures between 
the $L^r(v)$ and $L^{r'}(v^{-r'/r})$ boundedness of $g_\psi$ 
implies the $L^2$ boundedness of $g_\psi$. 
Thus we have $m\in L^\infty(\Bbb R^n)$. 
\par 
Let $F(y,t)=f*\psi_t(y)$,  $f\in L^p_w\cap L^2$.  Then 
\begin{equation*} 
E_{\widetilde{\bar{\psi}}}^\epsilon(F)(x)=\int_\epsilon^{\epsilon^{-1}}
\int_{\Bbb R^n} \psi_t*f(y)\bar{\psi}_t(y-x)\,dy\,\frac{dt}{t} 
= \int_{\Bbb R^n} \Psi^{(\epsilon)}(x-z)f(z)\, dz, 
\end{equation*} 
where 
\begin{equation*} 
\Psi^{(\epsilon)}(x)=
\int_\epsilon^{\epsilon^{-1}}\int_{\Bbb R^n}  
\psi_t(x+y)\bar{\psi}_t(y)\,dy\,\frac{dt}{t}. 
\end{equation*} 
We see that 
\begin{equation*} 
\widehat{\Psi^{(\epsilon)}}(\xi)
=\int_\epsilon^{\epsilon^{-1}}\hat{\psi}(t\xi)\widehat{\bar{\psi}}(-t\xi)
\,\frac{dt}{t}
=\int_\epsilon^{\epsilon^{-1}}|\hat{\psi}(t\xi)|^2 \,\frac{dt}{t}. 
\end{equation*} 
Thus 
\begin{equation*} 
\int_{\Bbb R^n} \Psi^{(\epsilon)}(x-z)f(z)\, dz=T_{m^{(\epsilon)}}f(x), 
\quad m^{(\epsilon)}(\xi) 
=\int_\epsilon^{\epsilon^{-1}}|\hat{\psi}(t\xi)|^2 \,\frac{dt}{t}.   
\end{equation*} 
 From Lemma $2.7$ and the $L^r_{v}$ boundedness of $g_\psi$ 
it follows that  
\begin{equation}\label{boundm} 
\|T_{m^{(\epsilon)}}f\|_{r,v}=\|E_{\widetilde{\bar{\psi}}}^\epsilon(F)
\|_{r,v}\leq C_0(r,v)\|g_\psi(f)\|_{r,v}\leq C_0(r,v)C_1(r,v)\|f\|_{r,v}.  
\end{equation} 
Letting $\epsilon\to 0$, we see that $m\in M^r(v)$ and 
$\|m\|_{M^r(v)}$ can be evaluated by \eqref{boundm}.  
\end{proof}  
Now we can state a weighted version of Theorem C.  
\begin{theorem}\label{weightLP}
 Let $g_\psi$ be as in \eqref{lpop}. Let $w\in A_p$, 
$1<p<\infty$.
Suppose that there exists a $(w,p)$ set $U(w,p)$ such that 
$g_\psi$ fulfills the hypotheses of Proposition $2.8$ on the weighted 
boundedness for all $r$, $v=w^s$, $(r,s)\in U(w,p)$.  
Further, suppose that 
$m(\xi)=\int_0^\infty|\hat{\psi}(t\xi)|^2\, dt/t$ is continuous and 
does not vanish on $S^{n-1}$. 
Then we have 
$$\|f\|_{p,w}\leq C_{p,w} \|g_\psi(f)\|_{p,w} $$  
for $f\in L^p_w$. 
\end{theorem} 
Obviously, this implies Theorem C when $w=1$. 
\begin{proof}[Proof of Theorem $2.9$] 
We first note that by Proposition $2.8$  
$m\in M(U(w,p))$.  Thus from  Theorem $2.5$ with $\varphi(z)=1/z$ 
and our assumptions, we see that $m^{-1}\in M^p(w)$. 
Since $f=T_{m^{-1}}T_m f$, $f\in L^p_w\cap L^2$, we have   
$$\|f\|_{p,w}\leq C\|T_mf\|_{p,w}. $$ 
Also, by \eqref{boundm} it follows that 
$$\|T_mf\|_{p,w}\leq C \|g_\psi(f)\|_{p,w}.$$ 
Combining results we have the desired inequality.  
\end{proof}
\par 
From Theorem \ref{weightLP} the next result follows. 
\begin{theorem}\label{weightLP2} 
Suppose the following. 
\begin{enumerate} 
\item[(1)] $\|g_\psi(f)\|_{r,v}\leq C_{r,v}\|f\|_{r,v}$ for all 
$r\in (1,\infty)$ and all $v\in A_r;$  
\item[(2)] $m(\xi)=\int_0^\infty|\hat{\psi}(t\xi)|^2\, dt/t$ is continuous and 
strictly positive on $S^{n-1}$. 
\end{enumerate}  
Then, if $f\in L^p_w$, we have 
$$\|f\|_{p,w}\leq C_{p,w} \|g_\psi(f)\|_{p,w}  $$  
 for all $p\in (1,\infty)$ and $w\in A_p$. 
\end{theorem} 
The following result is known (see \cite{Sa}). 
\begin{theoreme}  
Suppose that  
\begin{enumerate}
\renewcommand{\labelenumi}{(\arabic{enumi})} 
\item  $B_{\epsilon}(\psi) < \infty$  \quad  for some $\epsilon >0, $ 
where $B_{\epsilon}(\psi) = \int_{|x|>1}|\psi(x)|\,|x|^{\epsilon}\,dx; $   
\item  $C_u(\psi)<\infty$ \quad for some $u>1$ with $C_u(\psi)=
\int_{|x|<1}|\psi(x)|^u\,dx;$  
\item  $H_\psi \in L^1(\Bbb R^n),$ \quad where 
$H_{\psi}(x) = \sup_{|y|\geq |x|}|\psi(y)|$.  

\end{enumerate} 
Then 
$$\|g_\psi(f)\|_{p,w} \leq C_{p,w}\|f\|_{p,w}$$
 for all $p \in (1, \infty)$ and $w \in A_p$. 
\end{theoreme}   
By Theorem \ref{weightLP2} and Theorem E we have the following result, which 
is useful in some applications. 
\begin{corollary} Suppose that $\psi$ satisfies the conditions $(1), (2), (3)$ 
of Theorem E and the non-degeneracy condition$:$  
$\sup_{t>0}|\hat{\psi}(t\xi)|>0$  for all $\xi\neq 0$.  
Then  $\|f\|_{p,w}\simeq \|g_\psi(f)\|_{p,w}$, $f\in L^p_w$, 
  for all $p \in (1, \infty)$ and $w \in A_p$.
\end{corollary} 
\begin{proof} Let 
$m(\xi)=\int_0^\infty|\hat{\psi}(t\xi)|^2\, dt/t$. Then by our assumption 
$m(\xi)\neq 0$ for $\xi\neq 0$. Thus 
by Theorem E and Theorem \ref{weightLP2}, it suffices to show that $m$ is 
continuous on $S^{n-1}$. 
From \cite{Sa} it can be seen that 
$\int_\epsilon^{\epsilon^{-1}}|\hat{\psi}(t\xi)|^2\, dt/t \to m(\xi)$ uniformly on $S^{n-1}$ as $\epsilon\to 0$.  
Since $\int_\epsilon^{\epsilon^{-1}}|\hat{\psi}(t\xi)|^2\, dt/t $ is continuous  on $S^{n-1}$ for each fixed $\epsilon>0$, the continuity of $m$ follows.    
\end{proof}

\begin{remark}  
Let $1<p<\infty$, $w\in A_p$. Suppose that $g_\psi$ is bounded on $L^p_w$ 
and $\psi$ is a radial function with $\int_0^\infty |\hat{\psi}(t\xi)|^2\, 
dt/t=1$ for every $\xi\neq 0$. Then we have $\|f\|_{p,w}\leq 
C\|g_\psi(f)\|_{p,w}$ if $g_\psi$ is also bounded on $L^{p'}_{w^{-p'/p}}$.  
This is well-known and follows from the proofs of Lemma $2.7$ and Proposition 
$2.8$.  Also, this can be proved by applying arguments of 
\cite[Chap. V, $5.6$ (b)]{GR}.   
\end{remark}

\section{Discrete parameter Littlewood-Paley functions} 
  
Let $\psi \in L^1(\Bbb R^n)$ with \eqref{cancell} and  let 
$\Delta_\psi$ be as in \eqref{dlpop}. 
We first give a criterion for the boundedness of $\Delta_\psi$ on $L^p_w$ 
analogous to Theorem E.  
\begin{theorem} Let $B_{\epsilon}(\psi)$, $H_\psi$ be as in Theorem E.  
  Suppose that
\begin{enumerate}   
\item[(1)] $B_{\epsilon}(\psi) < \infty$ \quad   
for some $\epsilon>0;$  
\item[(2)] $|\hat{\psi}(\xi)|\leq C|\xi|^{-\delta}$ \quad for all 
$\xi\in \Bbb R^n\setminus\{0\}$ with some $\delta>0;$  
\item[(3)] $H_\psi\in L^1(\Bbb R^n)$. 
\end{enumerate}
Let $1<p<\infty$. Then 
$$\|\Delta_\psi(f)\|_{p,w} \leq C_{p,w}\|f\|_{p,w}$$
 for every $w\in A_p$. 
\end{theorem} 
We assume the pointwise estimate of $\hat{\psi}$ in $(2)$, which is not 
required in Theorem E. 
\begin{proof}[Proof of Theorem $3.1$]  
We apply methods of \cite{DR}.  Define 
$$\widehat{D_j(f)}(\xi) = \Psi(2^j\xi)\hat{f}(\xi) \qquad \text{for \quad $j \in \Bbb Z$},$$ 
where $\Psi \in C^{\infty}$ satisfies that $\supp(\Psi) \subset 
\{1/2\leq |\xi|\leq 2\}$ and  
$$\sum_{j=-\infty}^\infty \Psi(2^j\xi) = 1 \qquad \text{for \quad $\xi \neq 0$.}$$ 
 We write  
 $$f* \psi_{2^k}(x) = \sum_{j=-\infty}^\infty D_{j+k}(f* \psi_{2^k})(x),  
 $$   
 where we initially assume that $f\in \mathscr S(\Bbb R^n)$. 
 Let 
 $$L_j(f)(x) = \left(\sum_{k=-\infty}^\infty
 \left|D_{j+k}(f* \psi_{2^k})(x)\right|^2 \right)^{1/2}.$$ 
 Then 
 $$\Delta_\psi(f)(x) \leq  \sum_{j\in \Bbb Z}L_j(f)(x).$$
\par 
We note that the condition $(1)$ and \eqref{cancell} imply that 
$|\hat{\psi}(\xi)|\leq C|\xi|^{\epsilon'}$, $\epsilon'=\min(1,\epsilon)$. So, 
since the Fourier transform of $D_{j}(f* \psi_{2^k})$ is supported in 
 $E_j =\{2^{-1-j}\leq |\xi| \leq 2^{1-j}\}$,  the Plancherel theorem and 
the conditions $(1)$, $(2)$  imply  that 
\begin{align} 
 \|L_j(f)\|_2^2 &= \sum_{k\in \Bbb Z}\int_{\Bbb R^n}
 \left|D_{j+k}\left(f* \psi_{2^k} \right)(x)\right|^2 \, dx
 \\
 &\leq  \sum_{k\in \Bbb Z} C\int_{E_{j+k}}
 \min\left(|2^k\xi|^{\epsilon}, |2^k\xi|^{-\epsilon}\right) 
 \left|\hat{f}(\xi)\right|^2 \,d\xi     \notag
 \\
 &\leq C 2^{-\epsilon |j|} \sum_{k\in \Bbb Z}
 \int_{E_{j+k}}\left|\hat{f}(\xi)\right|^2 \,d\xi                     
         \notag 
 \\
 &\leq C 2^{-\epsilon |j|}\|f\|_2^2     \notag 
 \end{align}
for some $\epsilon>0$, 
 where to get the last inequality we also use the fact that the sets $E_j$ 
 are  finitely overlapping.  
 \par 
By the condition $(3)$, we see that $\sup_{t>0}|f*\psi_t|\leq 
CM(f)$. Thus, if $w\in A_2$, by the Hardy-Littlewood maximal theorem and 
  the Littlewood-Paley inequality for $L^2_w$  we see that   
 \begin{align} 
 \|L_j(f)\|_{2,w}^2 &= \sum_{k\in \Bbb Z}\int_{\Bbb R^n}
 \left|D_{j+k}(f)* \psi_{2^k} (x)\right|^2 \, w(x)\, dx 
 \\
 &\leq \sum_{k\in \Bbb Z} C\int_{\Bbb R^n}\left|M(D_{j+k}(f))(x)\right|^2w(x)
\, dx             \notag 
 \\
 &\leq \sum_{k\in \Bbb Z} C\int_{\Bbb R^n}\left|D_{j+k}(f)(x)\right|^2w(x)\, dx
 \notag 
 \\
 &\leq C\|f\|_{2,w}^2.   \notag 
 \end{align}
 \par 
 Interpolation with change of measures between $(3.1)$ and $(3.2)$ implies 
 that 
 $$\|L_j(f)\|_{2,w^u} \leq C2^{-\epsilon(1-u)|j|/2} \|f\|_{2,w^u}$$
 for $u\in (0, 1)$.  Choosing $u$, close to 1, so that $w^{1/u} \in A_2$, 
  we have  
  $$\|L_j(f)\|_{2,w} \leq C2^{-\epsilon(1-u)|j|/2} \|f\|_{2,w}, $$
  and hence 
  $$\|\Delta_\psi(f)\|_{2,w} \leq  \sum_{j\in \Bbb Z}\|L_j(f)\|_{2,w} \leq 
  C\|f\|_{2,w}.$$ 
  Thus the conclusion follows from the extrapolation theorem of 
  Rubio de Francia \cite{Ru}. 
\end{proof}
 
\begin{remark}  
Under the hypotheses of Theorem $3.1$, $g_\psi$ is also bounded on $L^p_w$ for 
all $p\in (1,\infty)$ and $w\in A_p$. This can be seen from the proof of 
Theorem E in \cite{Sa}. 
\end{remark}
\par   
Let $\mathscr K$ be the Hilbert space of 
functions $v(k)$ on $\Bbb Z$ such that 
$$\|v\|_{\mathscr K}
=\left(\sum_{k=-\infty}^\infty |v(k)|^2\right)^{1/2}<\infty. $$   
We define spaces $L^p_{w, \mathscr K}$, similarly to $L^p_{w, \mathscr H}$. 
Also, we use notation similar to the one used when $E_\psi^\epsilon(h)$ is 
considered.   
We define 
\begin{equation}  
L_\psi^N(l)(x)=\sum_{k\in \Bbb Z}\int_{\Bbb R^n}\psi_{2^k}(x-y)l_{(N)}(y,k)
\, dy,  
\end{equation} 
where $l\in L^2_\mathscr K$, $l_{(N)}(x,k)=l(x,k)\chi_{[-N,N]}(k)$ for a 
positive integer $N$. 
\par 
Then, we have the following result.   
\begin{lemma}  
Suppose that $1<r<\infty$, $v\in A_r$ and that  
$$\|\Delta_\psi(f)\|_{r',v^{-r'/r}}\leq C_0(r,v)\|f\|_{r',v^{-r'/r}}.   
$$ 
Then, we have 
 $\sup_{N\geq 1}\|L_{\widetilde{\bar{\psi}}}^N(l)\|_{r, v}\leq 
C_0(r,v)\|l\|_{r, v, \mathscr K}$, that is, 
$$\sup_{N\geq 1}\left(\int_{\Bbb R^n} |L_{\widetilde{\bar{\psi}}}^N
(l)(x)|^r v(x)\, dx \right)^{1/r}\leq 
C_0(r,v)\left(\int_{\Bbb R^n}\left(\sum_{k=-\infty}^\infty
|l(x,k)|^2\right)^{r/2}v(x)\, dx\right)^{1/r}   $$  
for $l\in L^r_{v, \mathscr K}\cap L^2_\mathscr K$.   
\end{lemma}  
This is used to prove the following. 
\begin{proposition}      
We assume that $\Delta_\psi$ satisfies   
the hypothesis of Lemma $3.3$ with  $r\in (1,\infty)$ and $v\in A_r$.  
Further, we assume that    
$$\|\Delta_\psi(f)\|_{r,v}\leq C_1(r,v)\|f\|_{r,v}.  $$ 
Set 
$$ m(\xi)= \sum_{k=-\infty}^\infty |\hat{\psi}(2^k\xi)|^2. $$  
Then, we have 
 $\|m\|_{M^r(v)}\leq C_0(r,v)C_1(r,v)$. 
\end{proposition} 
Proposition 3.4 and Theorem $2.5$ are applied to prove the following.   
\begin{theorem}\label{dweightLP}
We assume that $w\in A_p$, $1<p<\infty$.  Suppose  that 
there exists a $(w,p)$ set $U(w,p)$ such that 
 $\Delta_\psi$ fulfills the hypotheses of Proposition $3.4$ 
on the weighted boundedness for all $r$, $v=w^s$, $(r,s)\in U(w,p)$. 
Then if the function 
$m(\xi)= \sum_{k=-\infty}^\infty |\hat{\psi}(2^k\xi)|^2$ 
 is continuous and does not vanish on $B_0$,  we have 
$$\|f\|_{p,w}\leq C_{p,w} \|\Delta_\psi(f)\|_{p,w} $$   
for $f\in L^p_w$. 
\end{theorem} 
We note that $m$ is dyadically homogeneous of degree $0$ and 
that, under the assumptions of Theorem \ref{dweightLP},  $m\in M(U(w,p))$. 
\par 
Theorem \ref{dweightLP} implies the next result. 
\begin{theorem} 
We assume the following. 
\begin{enumerate} 
\item[(1)] $\|\Delta_\psi(f)\|_{r,v}\leq C_{r,v}\|f\|_{r,v}$ for all 
$r\in (1,\infty)$ and all $v\in A_r;$    
\item[(2)] $m$ is continuous and strictly positive on $B_0$, where $m$ is 
defined as in Theorem $3.5$. 
\end{enumerate}  
Let $w\in A_p$, $1<p<\infty$. Then we have 
$$\|f\|_{p,w}\leq C_{p,w} \|\Delta_\psi(f)\|_{p,w}, \quad f\in L^p_w.     $$  
\end{theorem} 

Lemma 3.3, Proposition 3.4, Theorem 3.5 and Theorem 3.6  are analogous 
to and can be proved similarly to Lemma 2.7, Proposition 2.8, Theorem 2.9 
and Theorem 2.10,  respectively. We omit their proofs. 
\par 
We also have an analogue of Corollary 2.11.  
\begin{corollary} Suppose that $\psi$ satisfies the conditions $(1), (2), (3)$ 
of Theorem $3.1$ and the non-degeneracy condition$:$ 
 $\sup_{k\in \Bbb Z}|\hat{\psi}(2^k\xi)|>0$ for all $\xi\neq 0$.  
Then  $\|f\|_{p,w}\simeq \|\Delta_\psi(f)\|_{p,w}$, $f\in L^p_w$, 
  for all $p \in (1, \infty)$ and $w \in A_p$.
\end{corollary} 
\begin{proof} By the assumption 
$m(\xi)=\sum_{k=-\infty}^\infty|\hat{\psi}(2^k\xi)|^2>0$ for $\xi\neq 0$. 
Therefore,     
by Theorem $3.6$, to prove a reverse inequality of the conclusion of 
Theorem $3.1$ it suffices to show that $m$ is continuous on 
$B_0$.  From the estimate $|\hat{\psi}(\xi)|\leq C\min(|\xi|^\epsilon, 
|\xi|^{-\epsilon}) $ for some $\epsilon>0$, which follows from  
(1) and (2) of Theorem $3.1$, it can be seen that 
$\sum_{k=-N}^N|\hat{\psi}(2^k\xi)|^2 \to m(\xi)$ uniformly on 
$B_0$ as $N \to \infty$.  
Since $\sum_{k=-N}^N|\hat{\psi}(2^k\xi)|^2$ is continuous  on 
$B_0$ for each fixed $N$, we can conclude that
 $m$ is also continuous on $B_0$.   
This completes the proof.  
\end{proof}

\section{Littlewood-Paley operators on $H^p$, $0<p\leq 1$, 
with $p$ close to $1$}  
   
Let $0<p\leq 1$. 
We consider the Hardy space of functions on $\Bbb R^n$ with values in 
$\mathscr H$, which is denoted by 
$H^p_{\mathscr H}(\Bbb R^n)$.  
Choose $\varphi\in \mathscr S(\Bbb R^n)$ with $\int \varphi(x)\, dx=1$. 
Let $h \in L^2_{\mathscr H}(\Bbb R^n)$.  
We say  $h \in H^p_{\mathscr H}(\Bbb R^n)$ if 
$\|h\|_{H^p_{\mathscr H}}= \|h^*\|_{L^p}<\infty$  
with 
$$h^*(x)=\sup_{s>0}\left(\int_0^\infty |\varphi_s*h^{t}(x)|^2\, \frac{dt}{t}
\right)^{1/2}, $$where we write $h^{t}(x)=h(x,t)$. 
Similarly, we consider the Hardy space  $H^p_{\mathscr K}(\Bbb R^n)$ of  
 functions $l$ in $ L^2_{\mathscr K}(\Bbb R^n)$ 
such that 
$\|l\|_{H^p_{\mathscr K}}^p= \|l^*\|_{L^p}<\infty$,   
where 
$$l^*(x)=\sup_{s>0}\left(\sum_{j=-\infty}^\infty |\varphi_s*l^{j}(x)|^2
\right)^{1/2}, \quad l^{j}(x)=l(x,j).   $$  
\par 
 Let $\psi\in L^1(\Bbb R^n)$ with \eqref{cancell} and let 
 $E_\psi^\epsilon(h)$ be defined as in $(2.4)$. 
\begin{theorem} 
 Suppose that 
\begin{enumerate} 
\item[(1)]  $\int_0^\infty |\hat{\psi}(t\xi)|^2 \, dt/t \leq C$ with a constant  $C;$   
\item[(2)]  there exists $\tau\in (0,1]$ such that if $|x|>2|y|$, 
$$\left(\int_0^\infty |\psi_t(x-y)-\psi_t(x)|^2 \, \frac{dt}{t}\right)^{1/2} 
\leq C\frac{|y|^\tau}{|x|^{n+\tau}}. $$
\end{enumerate}
Then 
$$\sup_{\epsilon\in (0,1)}
\|E_\psi^\epsilon(h)\|_{H^p}\leq C\|h\|_{H^p_{\mathscr H}} $$  
if  $n/(n+\tau) <p\leq 1$, where $H^p=H^p(\Bbb R^n)$ is the 
ordinary Hardy space on $\Bbb R^n$. 
\end{theorem} 
Recall that we say $f\in \mathscr S'(\Bbb R^n)$ (the space of tempered 
distributions) belongs to 
$H^p(\Bbb R^n)$ if 
$\|f\|_{H^p}=\|f^*\|_p<\infty$ , where $f^*(x)=\sup_{t>0}|\varphi_t*f(x)|$, 
with $\varphi\in \mathscr S(\Bbb R^n)$ satisfying $\int \varphi(x)\, dx=1$ 
(see \cite{FeS}).  
\par We also have a similar result for $L_\psi^N(l)$. 
\begin{theorem} 
Let $L_\psi^N(l)$ be defined as in $(3.3)$. 
We assume the following conditions$:$    
\begin{enumerate} 
\item[(1)]  $\sum_{k=-\infty}^\infty |\hat{\psi}(2^k\xi)|^2 \leq C$ with a constant  $C;$   
\item[(2)]   if $|x|>2|y|$, we have 
$$\left(\sum_{k=-\infty}^\infty |\psi_{2^k}(x-y)-\psi_{2^k}(x)|^2 
\right)^{1/2} 
\leq C\frac{|y|^\tau}{|x|^{n+\tau}} $$ 
with some $\tau\in (0,1]$.  
\end{enumerate}  
Then 
$$\sup_{N\geq 1}\|L_\psi^N(l)\|_{H^p}\leq C\|l\|_{H^p_{\mathscr K}}  
\quad \text{for $n/(n+\tau) <p\leq 1$.}     $$  
\end{theorem}  

To prove these theorems we apply atomic decompositions. 

Let $a$ be a $(p, \infty)$ atom in  $H^p_{\mathscr H}(\Bbb R^n)$. Thus 
\begin{enumerate} 
\item[(i)] $\left(\int_0^\infty |a(x,t)|^2\, dt/t \right)^{1/2} \leq 
|Q|^{-1/p}$, where $Q$ is a cube in $\Bbb R^n$ with sides parallel to the 
coordinate axes; 
\item[(ii)] $\sup(a(\cdot, t))\subset Q$ uniformly in $t>0$, where $Q$ is the
 same as in $(i);$  
\item[(iii)] $\int_{\Bbb R^n} a(x,t)x^\gamma \, dx=0$ for all $t>0$ and 
$|\gamma|\leq [n(1/p -1)]$, where 
$x^\gamma=x_1^{\gamma_1}\dots x_n^{\gamma_n}$ and  $[a]$ denotes the largest 
integer not exceeding $a$.  
\end{enumerate}  
To prove Theorem $4.1$ we use the following.  
\begin{lemma}
Let $h\in L^2_{\mathscr H}(\Bbb R^n)$. Suppose that
 $h\in H^p_{\mathscr H}(\Bbb R^n)$. 
Then there exist a sequence $\{a_k\}$ of 
$(p,\infty)$ atoms in $H^p_{\mathscr H}(\Bbb R^n)$ and a sequence 
$\{\lambda_k\}$ of positive numbers such that 
$h=\sum_{k=1}^\infty\lambda_k a_k$ in $H^p_{\mathscr H}(\Bbb R^n)$ and in 
$L^2_{\mathscr H}(\Bbb R^n)$,  and $\sum_{k=1}^\infty\lambda_k^p\leq 
C\|h\|_{H^p_{\mathscr H}}^p$, 
where $C$ is a constant independent of $h$.  
\end{lemma} 
See \cite{GR, ST} for the case of $H^p(\Bbb R^n)$; the vector valued case 
 can be proved similarly.  
We apply Lemma $4.3$ for $p\in (n/(n+1), 1]$.  
 We also need the following. 

\begin{lemma} 
Let $\varphi$ be a non-negative $C^\infty$ function on $\Bbb R^n$ supported 
in $\{|x|<1\}$ which satisfies $\int \varphi(x)\, dx=1$. Suppose 
that $\psi\in L^1(\Bbb R^n)$ satisfies the conditions $(1), (2)$ of 
Theorem $4.1$. 
Let $\Psi_{s,t}=\varphi_s*\psi_t$, $s, t>0$. Then, if $|x|>3|y|$, we have 
$$\left(\int_0^\infty |\Psi_{s,t}(x-y)-\Psi_{s,t}(x)|^2 
\, \frac{dt}{t}\right)^{1/2} \leq C\frac{|y|^\tau}{|x|^{n+\tau}} $$  
with a constant $C$ independent of $s>0$. 
\end{lemma} 
\begin{proof} 
We note that 
$$\Psi_{s,t}(x-y)-\Psi_{s,t}(x)=\int_{|z|<s}
(\psi_t(x-y-z)-\psi_t(x-z)) \varphi_s(z)\, dz.  $$ 
Let $0<s<|x|/4$. Then, if $|x|>3|y|$ and $|z|<s$, we have $|x-z|\geq 
(3/4)|x|\geq 2|y|$.  Thus by the Minkowski inequality and $(2)$ of Theorem $4.1$ we see that 
\begin{align} 
&\left(\int_0^\infty |\Psi_{s,t}(x-y)-\Psi_{s,t}(x)|^2 
\, \frac{dt}{t}\right)^{1/2} 
\\ 
&\leq \int_{|z|<s}\left(\int_0^\infty |\psi_t(x-y-z)-\psi_t(x-z)|^2 
\, \frac{dt}{t}\right)^{1/2}\varphi_s(z)\, dz        \notag 
\\ 
&\leq C\int_{|z|<s} \frac{|y|^\tau}{|x-z|^{n+\tau}}\varphi_s(z)\, dz \notag 
\\
&\leq C\|\varphi\|_1\frac{|y|^\tau}{|x|^{n+\tau}}. \notag 
\end{align} 
\par 
To deal with the case $s\geq |x|/4$, we write 
\begin{align*} 
\Psi_{s,t}(x-y)-\Psi_{s,t}(x)=\int \hat{\varphi}(s\xi)\hat{\psi}(t\xi)
e^{2\pi i\langle x,\xi\rangle}\left(e^{-2\pi i\langle y,\xi\rangle}-1 \right)\, d\xi. 
\end{align*} 
Applying Minkowski's inequality again and using $(1)$ of Theorem $4.1$, we see 
that 
\begin{align} 
&\left(\int_0^\infty |\Psi_{s,t}(x-y)-\Psi_{s,t}(x)|^2 
\, \frac{dt}{t}\right)^{1/2} 
\\ 
&\leq \int |\hat{\varphi}(s\xi)|2\pi |y||\xi|\left(\int_0^\infty 
|\hat{\psi}(t\xi)|^2\, \frac{dt}{t}\right)^{1/2}\, d\xi    \notag 
\\ 
&\leq C|y|\int |\hat{\varphi}(s\xi)| |\xi| \, d\xi \notag 
\\ 
&\leq C|y|s^{-n-1} \int |\hat{\varphi}(\xi)| |\xi| \, d\xi   \notag 
\\
&\leq C\frac{|y|^\tau}{|x|^{n+\tau}}. \notag 
\end{align} 
By $(4.1)$ and $(4.2)$ we get the desired estimates. 
\end{proof} 
\begin{proof}[Proof of Theorem $4.1$]  

Let $a$ be a $(p, \infty)$ atom in  $H^p_{\mathscr H}(\Bbb R^n)$ supported 
on the cube $Q$ of the definition of the atom. 
Let $y_0$ be the center of $Q$.  Let $\widetilde{Q}$ be a concentric enlargement of $Q$ such that $3|y-y_0|<|x-y_0|$  if $y\in Q$ and $x\in \Bbb R^n\setminus 
\widetilde{Q}$.  
Let $\varphi$ be as in Lemma $4.4$.  
Then, using Lemma $4.4$,  the properties of an atom and the Schwarz 
inequality, for $x\in \Bbb R^n\setminus \widetilde{Q}$ we have 
\begin{align*} 
\left|\varphi_s*E_\psi^\epsilon(a)(x)\right|&=\left|\iint 
\left(\Psi_{s,t}(x-y)-\Psi_{s,t}(x-y_0) \right)
a_{(\epsilon)}(y,t)\, dy\, \frac{dt}{t}\right|  
\\ 
&\leq \int_Q\left(\int_0^\infty\left|\Psi_{s,t}(x-y)-\Psi_{s,t}(x-y_0) 
\right|^2 \, \frac{dt}{t}\right)^{1/2}\left(\int_0^\infty|a(y,t)|^2
\, \frac{dt}{t}\right)^{1/2}\, dy
\\ 
&\leq C|Q|^{-1/p}\int_{Q}\left(\int_0^\infty
\left|\Psi_{s,t}(x-y)-\Psi_{s,t}(x-y_0) \right|^2 \, \frac{dt}{t}\right)^{1/2}\, dy 
\\ 
&\leq C|Q|^{-1/p}\int_{Q}|y-y_0|^\tau |x-y_0|^{-n-\tau}
\, dy 
\\ 
&\leq C|Q|^{-1/p+ 1+\tau/n}|x-y_0|^{-n-\tau}. 
\end{align*}  
Since $p>n/(n+\tau)$, we thus have  
\begin{equation}\label{a1}
\int_{\Bbb R^n\setminus \widetilde{Q}} \sup_{s>0}
\left|\varphi_s*E_\psi^\epsilon(a)(x)
\right|^p\, dx 
\leq C|Q|^{-1+ p+p\tau/n}\int_{\Bbb R^n\setminus \widetilde{Q}}
|x-y_0|^{-p(n+\tau)} \leq C. 
\end{equation} 
\par 
The condition $(1)$ implies the $L^2$ boundedness of $g_\psi$ and hence  
by Lemma 2.7 we can see that 
$$\sup_{\epsilon\in(0,1)}\|E_\psi^\epsilon(h)\|_2 
\leq C\|h\|_{L^2_\mathscr H}, \quad h\in L^2_{\mathscr H}(\Bbb R^n).$$ 
So, by H\"{o}lder's inequality and the properties (i), (ii) of $a$, we get 
\begin{align}\label{a2}
\int_{\widetilde{Q}} \sup_{s>0}\left|\varphi_s*E_\psi^\epsilon(a)(x)
\right|^p\, dx 
&\leq C|Q|^{1-p/2}\left(\int_{\widetilde{Q}}|M(E_\psi^\epsilon(a))(x)|^2
\, dx \right)^{p/2} 
\\ 
&\leq C|Q|^{1-p/2}\left(\int_{Q}\int_0^\infty |a(y,t)|^2\, \frac{dt}{t} \, dy 
\right)^{p/2}    \notag 
\\ 
&\leq C.    \notag
\end{align} 
Combining \eqref{a1} and \eqref{a2}, we have 
\begin{equation}\label{a3}
\int_{\Bbb R^n} \sup_{s>0}\left|\varphi_s*E_\psi^\epsilon(a)(x)\right|^p
\, dx \leq C. 
\end{equation} 
By Lemma $4.3$ and $(4.5)$ we can prove 
 \begin{equation*}\label{h1}
\int_{\Bbb R^n} \sup_{s>0}\left|\varphi_s*E_\psi^\epsilon(h)(x)\right|^p
\, dx \leq C\|h\|_{H^p_{\mathscr H}}^p.   
\end{equation*} 
This completes the proof.

\end{proof} 
Theorem $4.2$ can be shown similarly.  
\par    
Also, we can prove the following mapping properties of $g_\psi$ and 
$\Delta_\psi$ on $H^p(\Bbb R^n)$ in the same way. 
\begin{theorem} Suppose that $\psi$ fulfills the hypotheses of 
Theorem $4.1$.  Set $F(\psi,f)(x,t)=f*\psi_t(x)$.  
Then if  $n/(n+\tau) <p\leq 1$, 
$$\|F(\psi, f)\|_{H^p_{\mathscr H}}\leq C\|f\|_{H^p}  $$  
for $f\in H^p(\Bbb R^n)\cap L^2(\Bbb R^n)$.   
\end{theorem}
\begin{theorem} We assume that $\psi$ fulfills the hypotheses of 
Theorem $4.2$.  Let $G(\psi,f)(x,k)=f*\psi_{2^k}(x)$.
Then 
$$\|G(\psi,f)\|_{H^p_{\mathscr K}}\leq C\|f\|_{H^p}, \quad 
f\in H^p(\Bbb R^n)\cap L^2(\Bbb R^n),  $$  
if  $n/(n+\tau) <p\leq 1$. 
\end{theorem}
\begin{proof}[Proof of Theorem $4.5$]  The proof is similar to that 
of  Theorem $4.1$.  By the atomic decomposition, it suffices 
to show that  $\|F(\psi, a)\|_{H^p_{\mathscr H}}\leq C$, where $a$ is a 
$(p,\infty)$ atom in $H^p(\Bbb R^n)$ such that $\|a\|_\infty\leq |Q|^{-1/p}$, 
$\supp(a)\subset Q$ with a cube $Q$ and $\int a=0$. Let $y_0$ be the center of 
$Q$ and 
let $\widetilde{Q}$, $\varphi_s$, $\Psi_{s,t}$ be as in the proof of 
Theorem $4.1$. Then, using Minkowski's inequality and Lemma $4.4$, for  
$x\in \Bbb R^n\setminus \widetilde{Q}$  we have 
\begin{align*} 
\left(\int_0^\infty
\left|\varphi_s*\psi_t*a(x)\right|^2\, \frac{dt}{t}\right)^{1/2} 
&=\left(\int_0^\infty\left|\int 
\left(\Psi_{s,t}(x-y)-\Psi_{s,t}(x-y_0) \right)
a(y)\, dy\right|^2  \, \frac{dt}{t}\right)^{1/2} 
\\ 
&\leq C|Q|^{-1/p}\int_{Q}\left(\int_0^\infty
\left|\Psi_{s,t}(x-y)-\Psi_{s,t}(x-y_0) \right|^2 \, \frac{dt}{t}\right)^{1/2}\, dy 
\\ 
&\leq C|Q|^{-1/p+ 1+\tau/n}|x-y_0|^{-n-\tau}. 
\end{align*}  
Therefore, as in \eqref{a1}, for $p>n/(n+\tau)$, we have 
\begin{equation*}
\int_{\Bbb R^n\setminus \widetilde{Q}} \sup_{s>0}
\left(\int_0^\infty
\left|\varphi_s*\psi_t*a(x)\right|^2\, \frac{dt}{t}\right)^{p/2} 
\leq C. 
\end{equation*}  
Since by the Minkowski inequality we easily see that
$$ \sup_{s>0}
\left(\int_0^\infty
\left|\varphi_s*\psi_t*a(x)\right|^2\, \frac{dt}{t}\right)^{1/2} 
\leq \sup_{s>0}\varphi_s*g_\psi(a)(x)\leq CM(g_\psi(a))(x),  $$  
 as in \eqref{a2} we have 
\begin{equation*}
\int_{\widetilde{Q}} \sup_{s>0}
\left(\int_0^\infty
\left|\varphi_s*\psi_t*a(x)\right|^2\, \frac{dt}{t}\right)^{p/2} 
\leq C. 
\end{equation*}  
Collecting results, we have the desired estimates. 

\end{proof} 
The proof of Theorem 4.6 is similar. 
Using Theorems 4.1, 4.5 and Theorems 4.2, 4.6, we can show analogues of 
Corollaries 2.11 and 3.7 for $p\leq 1$.  
\begin{theorem} Suppose that $\psi$ fulfills the hypotheses of 
Theorem $4.1$. 
Put $m(\xi)=\int_0^\infty |\hat{\psi}(t\xi)|^2\, dt/t$. 
We assume that $m$ does not vanish in 
$\Bbb R^n\setminus\{0\}$ and  $m\in 
C^k(\Bbb R^n\setminus\{0\})$, where $k$ is a positive integer satisfying 
$k/n>1/p-1/2$, with $ n/(n+\tau)<p\leq 1$. Then we have 
$$\|F(\psi,f)\|_{H^p_{\mathscr H}}\simeq \|f\|_{H^p}$$ 
 for $f\in H^p(\Bbb R^n)\cap L^2(\Bbb R^n)$, where $F(\psi,f)$ is as in 
 Theorem $4.5$.  
\end{theorem} 
\begin{theorem} We assume that $\psi$ fulfills the hypotheses of 
Theorem $4.2$. 
Set $m(\xi)=\sum_{j=-\infty}^\infty 
|\hat{\psi}(2^j\xi)|^2$.  Let $ n/(n+\tau)<p\leq 1$. 
We  assume that $m(\xi)\neq 0$ for all  
$\xi\in \Bbb R^n\setminus\{0\}$ and  $m\in 
C^k(\Bbb R^n\setminus\{0\})$ with  a positive integer $k$ as in Theorem 
$4.7$. Let $G(\psi,f)$ be as in Theorem $4.6$. Then we have 
$$\|G(\psi,f)\|_{H^p_{\mathscr K}}\simeq \|f\|_{H^p}, \quad 
f\in H^p(\Bbb R^n)\cap L^2(\Bbb R^n). $$ 
\end{theorem}
\begin{proof}[Proof of Theorem $4.7$]  
By Theorem $4.5$ we have 
$\|F(\psi,f)\|_{H^p_{\mathscr H}}\leq C \|f\|_{H^p}$. 
 To prove the 
reverse inequality we note that $f=T_{m^{-1}}T_m f$.  Since $m^{-1} \in 
C^k(\Bbb R^n\setminus\{0\})$ and it is homogeneous of degree $0$, 
$m^{-1}$ is a Fourier multiplier for $H^p$ by \cite[pp. 347--348]{GR}. Thus 
\begin{equation}\label{eq1}
\|f\|_{H^p}\leq C\|T_m f\|_{H^p}
\leq C\liminf_{\epsilon\to 0}\|T_{m^{(\epsilon)}}f\|_{H^p},    
\end{equation}  and by the proof of Proposition $2.8$  we see that 
\begin{equation*} 
T_{m^{(\epsilon)}}f=E_{\widetilde{\bar{\psi}}}^\epsilon(F),    
\end{equation*} 
where $m^{(\epsilon)}$, $F$ are defined as in the proof of 
Proposition $2.8$. 
Thus by Theorem $4.1$ we have 
\begin{align*} 
\|T_{m^{(\epsilon)}}f\|_{H^p}=\|E_{\widetilde{\bar{\psi}}}^\epsilon(F)\|_{H^p}
\leq C\|F(\psi,f)\|_{H^p_{\mathscr H}}, 
\end{align*} 
which combined with \eqref{eq1} implies the reverse inequality. 
\end{proof} 
Theorem $4.8$ can be proved similarly. 
\par 
We note that Theorems 4.5 and 4.6 imply that $\|g_\psi(f)\|_p 
\leq C\|f\|_{H^p}$, $\|\Delta_\psi(f)\|_p \leq C\|f\|_{H^p}$.  
Under the assumptions of Theorems 4.7 and 4.8,  the reverse 
inequalities,  which would improve results, are not available 
at present stage of the research. 
 For related results which can handle Littlewood-Paley operators 
 like $g_Q$, we refer to \cite{U}.   
\par 
Let $\varphi^{(\alpha)}$ on $\Bbb R^1$ be as in \eqref{1.4}.   
Then we can show that 
\begin{equation}\label{mar}
\left(\int_0^\infty|\varphi^{(\alpha)}_t(x-y)-\varphi^{(\alpha)}_t(x)|^2
\, \frac{dt}{t}\right)^{1/2} 
\leq C\frac{|y|^\sigma}{|x|^{1+\sigma}}, \quad \sigma=(2\alpha-1)/2,  
\end{equation}  
if $2|y|<|x|$, where $1/2< \alpha<3/2$. Also, it is not difficult to see that 
the condition $(1)$ of Theorem $4.1$ is valid for $\varphi^{(\alpha)}$. Thus, 
from Theorem 4.5 we in particular have the second inequality of \eqref{1.5}
 for $1/2< \alpha<3/2$, $2/(2\alpha+1)<p\leq 1$.  
We shall give a proof of the estimate \eqref{mar} in Section $6$ for 
completeness.

\section{Applications to the theory of Sobolev spaces}  

Let $0<\alpha<n$ and 
\begin{equation}\label{cb2} 
T_\alpha(f)(x)=\left(\int_0^\infty \left|I_\alpha(f)(x)-
\dashint_{B(x,t)}I_\alpha(f)(y)\, dy\right|^2 
\frac{dt}{t^{1+2\alpha}}\right)^{1/2},  
\end{equation}  
where $I_\alpha$ is the Riesz potential operator defined by 
\begin{equation}\label{1.04}
\widehat{I_\alpha(f)}(\xi)=(2\pi|\xi|)^{-\alpha} \hat{f}(\xi).  
\end{equation}  
Then, from \cite{AMV} we can see the following result. 
\begin{theoremf}  Suppose that $1<p<\infty$ and $n\geq 2$. 
Let $T_\alpha$ be as in \eqref{cb2}.  Then 
$$\|T_1 (f)\|_p \simeq \|f\|_p, \quad f\in \mathscr S(\Bbb R^n).   $$
\end{theoremf}  
In \cite{AMV} this was used to prove Theorem D in Section 1 when $n\geq 2$.  
Theorem F is generalized to the weighted $L^p$ spaces 
(see \cite{HL, Sa4}).  
\par  
We consider square functions generalizing $U_\alpha$ and $T_\alpha$ in 
\eqref{cb1} and \eqref{cb2}.  
 Let 
 \begin{equation}\label{cev2} 
 U_\alpha (f)(x)=\left(\int_0^\infty \left|f(x)-
 \Phi_t*f(x)\right|^2 \frac{dt}{t^{1+2\alpha}}\right)^{1/2},  \quad \alpha>0, 
 \end{equation} 
with $\Phi \in \mathscr M^\alpha$,  where  we say 
$\Phi \in \mathscr M^\alpha$, $\alpha>0$,  
if $\Phi$ is a bounded function on $\Bbb R^n$ with compact support satisfying  
$\int_{\Bbb R^n} \Phi(x)\, dx=1$; 
if $\alpha\geq 1$, we further assume that 
\begin{equation}\label{moment} 
\int_{\Bbb R^n} \Phi(x)x^\gamma \, dx=0  \quad  \text{for all $\gamma$ with 
$1\leq |\gamma|\leq [\alpha]$.}   
\end{equation}
When $1\leq \alpha<2$, \eqref{moment} is satisfied if $\Phi$ is even; 
in particular,  $\chi_0=\chi_{B(0,1)}/|B(0,1)|\in \mathscr M^\alpha$ 
for  $0<\alpha<2$ and if $\Phi=\chi_0$ in \eqref{cev2}, we have  
$U_\alpha$ of \eqref{cb1}.  
\par 
We also consider 
 \begin{equation}\label{cev} 
 T_\alpha (f)(x)=\left(\int_0^\infty \left|I_\alpha(f)(x)-
 \Phi_t*I_\alpha(f)(x)\right|^2 \frac{dt}{t^{1+2\alpha}}\right)^{1/2},  
\quad 0<\alpha<n, 
 \end{equation} 
where  $\Phi\in \mathscr M^\alpha$. 
If we set $\Phi=\chi_0$ in \eqref{cev},  we get  $T_\alpha$ of \eqref{cb2}.   
 \par 
We prove the following.  
\begin{theorem} Suppose that $T_\alpha$ is as in \eqref{cev} and $0<\alpha<n$, 
$1<p<\infty$. Let   $w\in A_p$.   
  Then 
$$\|T_\alpha (f)\|_{p,w} \simeq \|f\|_{p,w}, \quad f\in 
\mathscr S(\Bbb R^n).     $$  
\end{theorem}  
By Theorem 5.1 we see that $U_\alpha$ 
can be used to characterize the weighted Sobolev spaces.  
\par  
Let $J_\alpha$ be the Bessel potential operator defined as 
$J_\alpha(g)=K_\alpha*g$  with 
$$\hat{K}_\alpha(\xi)=(1+4\pi^2|\xi|^2)^{-\alpha/2} $$  
(see \cite{St3}).   
Let $1<p<\infty$, $\alpha>0$ and $w\in A_p$. 
The weighted Sobolev 
space $W^{\alpha, p}_w(\Bbb R^n)$ is defined to be the collection of all the 
functions $f$ which can be expressed as $f=J_\alpha(g)$ with  
$g\in L^p_w(\Bbb R^n)$ and its norm is defined by  
$\|f\|_{p,\alpha,w}=\|g\|_{p,w}$.  
The weighted $L^p$ norm inequality for the Hardy-Littlewood maximal 
operator with $A_p$ weights (see \cite{GR}) implies that 
 $J_\alpha(g)\in L^p_w$ if $g\in L^p_w$, since it is known that 
 $|J_\alpha(g)|\leq CM(g)$  (see \cite{St3, Str}).  We also note that 
 $J_\alpha$ is injective on $L^p_w$.  So, the norm $\|f\|_{p,\alpha,w}$ is 
 well-defined.   
\par 
Applying Theorem 5.1, we have the following. 
\begin{corollary} Let $1<p<\infty$, $w\in A_p$ and $0<\alpha<n$.  
Let $U_\alpha$ be as in \eqref{cev2}.  Then 
$f\in W^{\alpha, p}_w(\Bbb R^n)$ if and only if $f\in L^p_w$ and 
$U_\alpha(f) \in L^p_w;$  furthermore, 
$$\|f\|_{p,\alpha,w}\simeq \|f\|_{p,w}+\|U_\alpha(f)\|_{p,w}.  $$
\end{corollary}  
For the case $n=1$ and $\alpha=1$, see Remark \ref{one} below. 
We refer to \cite{St, St2, Str, Wh} for relevant results. 
See \cite{HL} for characterization of the weighted Sobolev space $W^{1,p}_w$ 
using square functions. 
\par  
Also, we consider discrete parameter versions of  $T_\alpha$ and $U_\alpha$: 
\begin{gather}\label{dev}  
D_\alpha(f)=\left(\sum_{k=-\infty}^\infty\left|I_\alpha(f)(x)-
\Phi_{2^k}*I_\alpha(f)(x)\right|^2 2^{-2k\alpha} \right)^{1/2}, 
\quad 0<\alpha<n, 
\\
\label{dev2}  
E_\alpha(f)=\left(\sum_{k=-\infty}^\infty\left|f(x)-
\Phi_{2^k}*f(x)\right|^2 2^{-2k\alpha} \right)^{1/2}, \quad \alpha>0, 
\end{gather}   
where $\Phi\in \mathscr M^\alpha$.  If we put 
$\Phi=\chi_0$ in \eqref{dev2}, we have $E_\alpha$  
 of \eqref{db1}. 
We have discrete parameter analogues of Theorem $5.1$ and Corollary $5.2$. 
\begin{theorem} Let $0<\alpha<n$ and $1<p<\infty$. Let $D_\alpha$ be 
as in \eqref{dev}.  Then 
$$\|D_\alpha(f)\|_{p,w} \simeq \|f\|_{p,w}, \quad f\in \mathscr S(\Bbb R^n), $$
where $w$ is any weight in $A_p$. 
\end{theorem}  
\begin{corollary} Let $E_\alpha$ be as in \eqref{dev2}.  Suppose that 
 $1<p<\infty$, $w\in A_p$ and $0<\alpha<n$.  
  Then 
$f\in W^{\alpha, p}_w(\Bbb R^n)$ if and only if $f\in L^p_w$ and 
$E_\alpha(f) \in L^p_w;$ also,  
$$\|f\|_{p,\alpha,w}\simeq \|f\|_{p,w}+\|E_\alpha(f)\|_{p,w}.  $$
\end{corollary}  
\par 
A version of Theorem $5.1$ for $0<\alpha<2$ and $n\geq 2$ is shown in 
\cite{Sa4}, where $\Phi$ is assumed 
to be radial. Combining  the arguments of \cite{Sa4} with Corollary 2.11, 
we can relax the assumption that $\Phi$ is radial. 
\par 
Here we give  proofs of Theorem 5.3 and Corollary 5.4; Theorem 5.1 and 
Corollary 5.2 can be shown similarly. 
\begin{proof}[Proof of Theorem $5.3$] 
Recall that $\widehat{L}_\alpha(\xi)=(2\pi|\xi|)^{-\alpha}$, $0<\alpha<n$, 
if $L_\alpha(x)=\tau(\alpha)|x|^{\alpha-n}$ with 
$$\tau(\alpha)=\frac{\Gamma\left(n/2-\alpha/2\right)}
{\pi^{n/2}2^\alpha \Gamma\left(\alpha/2\right)}. $$ 
Let 
$$\psi(x)=L_\alpha(x)-\Phi*L_\alpha(x).    $$  
Then, we have $D_\alpha (f) =\Delta_\psi(f)$, $f\in  
\mathscr S(\Bbb R^n)$, by homogeneity of $L_\alpha$, where $D_\alpha$ is 
as in \eqref{dev}.   
We observe that $\psi$ can be written as  
\begin{equation}\label{dp1}
\psi(x)=\int \left(L_\alpha(x)-L_\alpha(x-y)\right)\Phi(y)\, dy.  
\end{equation}  
Because $\Phi$ is  bounded and  
compactly supported and $L_\alpha$ is locally integrable, we see that 
$$\sup_{|x|\leq 1} \left|\int L_\alpha(x-y)\Phi(y)\, dy\right|\leq C  
$$ 
for some constant $C$.  Using this inequality in the definition of $\psi$, 
we have 
\begin{equation}\label{dp2}
|\psi(x)|\leq C|x|^{\alpha-n} \quad \text{for $|x|\leq 1$.} 
\end{equation}  
By applying Taylor's formula and \eqref{moment}, 
we can easily deduce from \eqref{dp1} 
that 
\begin{equation}\label{dp3}
|\psi(x)|\leq C|x|^{\alpha-n-[\alpha]-1} \quad \text{for $|x|\geq 1$.} 
\end{equation}  
\par 
Taking the Fourier transform, we see that 
\begin{equation}\label{dp4}
\hat{\psi}(\xi)=(2\pi|\xi|)^{-\alpha}\left(1-\hat{\Phi}(\xi)\right).      
\end{equation}  
By \eqref{moment} this implies 
$|\hat{\psi}(\xi)|\leq C|\xi|^{[\alpha]+1-\alpha}$, from which 
the condition \eqref{cancell} follows, since $[\alpha]+1-\alpha>0$. 
It is easy to see that 
the conditions (1), (2) and (3) of Theorem $3.1$   
follow from the estimates \eqref{dp2}, \eqref{dp3} and \eqref{dp4}. 
Also, obviously we have 
$\sup_{k\in \Bbb Z}|\hat{\psi}(2^k\xi)|>0$ for all $\xi\neq 0$. 
Thus we can apply  Corollary $3.7$ to get the equivalence of the $L^p_w$ norms 
claimed. 
\end{proof}
\begin{proof}[Proof of Corollary $5.4$]  
Riesz potentials and Bessel potentials are related as follows.   
\begin{lemma} 
Let $\alpha>0$, $1<p<\infty$ and $w\in A_p$. 
\begin{enumerate} \renewcommand{\labelenumi}{(\arabic{enumi})}   
\item We have 
$$(2\pi|\xi|)^\alpha=\ell(\xi)(1+4\pi^2|\xi|^2)^{\alpha/2} $$ 
with a Fourier multiplier $\ell$ for $L^p_w$.  
\item  There exists a Fourier multiplier $m$ for $L^p_w$ such that
$$(1+4\pi^2|\xi|^2)^{\alpha/2}=m(\xi)+m(\xi)(2\pi|\xi|)^\alpha. $$ 
\end{enumerate} 
\end{lemma}  
To prove this we note that 
$$|\partial_\xi^\gamma \ell(\xi)|\leq C_\alpha |\xi|^{-|\gamma|}, \quad 
\xi\in \Bbb R^n\setminus\{0\},  $$ 
for all multi-indices $\gamma$  and similar estimates for $m(\xi)$. So, 
by a theorem on Fourier multipliers  for $L^p_w$ 
we can get the results as claimed (see \cite{CF, KW}).  
See also \cite[Lemma 4]{St2}.  
\par 
 When $g\in L^p_w$, $w\in A_p$,  $1<p<\infty$ and $0<\alpha<n$, we show that 
\begin{equation}\label{4.1}  
\|E_\alpha(J_\alpha(g))\|_{p,w}+\|J_\alpha(g)\|_{p,w}\simeq \|g\|_{p,w}.  
\end{equation}
We first prove \eqref{4.1} for $g\in \mathscr{S}_0(\Bbb R^n)$.  
Since $E_\alpha(J_\alpha(g))=D_\alpha(I_{-\alpha}J_\alpha(g))$ and 
$I_{-\alpha}J_\alpha(g)\in \mathscr{S}(\Bbb R^n)$, 
when $g\in \mathscr{S}_0(\Bbb R^n)$,  by Theorem 5.3 we have 
\begin{equation}\label{4.2} 
\|E_\alpha(J_\alpha(g))\|_{p,w}\simeq \|I_{-\alpha}J_\alpha(g)\|_{p,w},  
\end{equation}    
where $I_{-\alpha}$ is defined by \eqref{1.04} with $-\alpha$ in place of 
$\alpha$.    
Part (1) of Lemma 5.5 implies that 
\begin{equation*}
\|I_{-\alpha}J_\alpha(g)\|_{p,w}\leq C\|g\|_{p,w}  
\end{equation*} 
and hence 
\begin{equation}\label{4.3} 
\|E_\alpha(J_\alpha(g))\|_{p,w}\leq C\|g\|_{p,w}.    
\end{equation} 
On the other hand, by part (2) of Lemma 5.5 and \eqref{4.2} we have 
\begin{align}\label{4.4} 
\|g\|_{p,w}&=\|J_{-\alpha}J_\alpha(g)\|_{p,w}\leq C\|J_\alpha(g)\|_{p,w}
+C\|I_{-\alpha}J_\alpha(g)\|_{p,w} 
\\ 
&\leq C\|J_\alpha(g)\|_{p,w}+C\|E_\alpha(J_\alpha(g))\|_{p,w},     \notag 
\end{align} 
where we recall that the Bessel potential operator $J_\beta$ is defined 
on $\mathscr{S}(\Bbb R^n)$ for any $\beta\in \Bbb R$ by 
$\widehat{J_\beta(f)}(\xi)=(1+4\pi^2|\xi|^2)^{-\beta/2} \hat{f}(\xi)$.  
Also we have 
\begin{equation}\label{4.5} 
\|J_\alpha(g)\|_{p,w}\leq C\|M(g)\|_{p,w}\leq C\|g\|_{p,w}.    
\end{equation} 
Combining \eqref{4.3}, \eqref{4.4} and \eqref{4.5}, we have \eqref{4.1} 
for $g\in \mathscr{S}_0(\Bbb R^n)$. 
\par 
Now we show that \eqref{4.1} holds for any $g\in L^p_w$.     
 For a positive integer $N$, let  
$$E_\alpha^{(N)}(f)(x)=\left(\sum_{k=-N}^N\left|f(x)-
\Phi_{2^k}*f(x)\right|^2 2^{-2k\alpha} \right)^{1/2}. 
$$  
Then $E_\alpha^{(N)}(f)\leq C_N M(f)$, which  implies that $E_\alpha^{(N)}$ 
is bounded on $L^p_w$.  
We can take a sequence $\{g_k\}$ in $\mathscr{S}_0(\Bbb R^n)$ such that 
$g_k \to g$ in $L^p_w$ and $J_\alpha(g_k) \to J_\alpha(g)$ in $L^p_w$ as $k\to 
\infty$.  
By \eqref{4.1} for $\mathscr{S}_0(\Bbb R^n)$ we see that 
\begin{equation*} 
\|E_\alpha^{(N)}(J_\alpha(g_k))\|_{p,w}\leq C \|g_k\|_{p,w}.  
\end{equation*}
Letting $k\to \infty$, by $L^p_w$ boundedness and sublinearity of 
$E_\alpha^{(N)}$   we have 
\begin{equation*} 
\|E_\alpha^{(N)}(J_\alpha(g))\|_{p,w}\leq C \|g\|_{p,w}.  
\end{equation*} 
Thus, letting $N\to \infty$, we get 
\begin{equation*} 
\|E_\alpha(J_\alpha(g))\|_{p,w}\leq C \|g\|_{p,w}.  
\end{equation*} 
Therefore, we have 
\begin{align*}\label{4.7} 
&\lim_{k\to \infty} \left\| E_\alpha(J_\alpha(g))-
E_\alpha(J_\alpha(g_k))\right\|_{p,w}  
\leq  \lim_{k\to \infty}
\left\| E_\alpha(J_\alpha(g-g_k))\right\|_{p,w}  
\\            
&\leq C\lim_{k\to \infty} \left\|g-g_k\right\|_{p,w} 
=0.    
\end{align*}  
Consequently, letting $k\to \infty$ in the relation 
\begin{equation*}
\|E_\alpha(J_\alpha(g_k))\|_{p,w}+\|J_\alpha(g_k)\|_{p,w}
\simeq \|g_k\|_{p,w},   
\end{equation*}
which we have already proved,  
 we can obtain \eqref{4.1} for any $g\in L^p_w$. 
\par 
To complete the proof of Corollary 5.4,  it thus only remains to show that 
$f\in W^{\alpha, p}_w(\Bbb R^n)$ if $f\in L^p_w$ and $E_\alpha(f) \in L^p_w$.  
To prove this it is convenient to note the following.  
\begin{lemma}  
Suppose that $f\in L^p_w$, $w\in A_p$, $1<p<\infty$, $g\in \mathscr S(\Bbb R^n)$ and $\alpha>0$.  Then we have the following. 
\begin{enumerate} 
\item[$(1)$] $K_\alpha*(f*g)(x)=(K_\alpha*f)*g(x)=(K_\alpha*g)*f(x)$ for 
every $x\in \Bbb R^n;$  
\item[$(2)$] $\int_{\Bbb R^n}(K_\alpha*f)(y)g(y)\, dy= 
\int_{\Bbb R^n}(K_\alpha*g)(y)f(y)\, dy.$    
\end{enumerate} 
\end{lemma} 
\begin{proof} 
To prove part $(1)$,  by Fubini's theorem it suffices to show that 
$$I=\iint K_\alpha(x-z-y)|f(y)||g(z)|\, dy\, dz <\infty.   $$  
This is obvious, for 
\begin{align*} 
I&\leq C\int M(f)(x-z)|g(z)|\, dz=C\int M(f)(z)|g(x-z)|\, dz 
\\ 
&\leq C\|M(f)\|_{p,w}\left(\int |g(x-z)|^{p'}w(z)^{-p'/p}\, dz\right)^{1/p'} 
\\ 
&\leq C \|f\|_{p,w}\left(\int |g(x-z)|^{p'}w(z)^{-p'/p}\, dz\right)^{1/p'}, 
\end{align*} 
where the last integral is finite since $g\in \mathscr S(\Bbb R^n)$ and 
$w^{-p'/p}\in A_{p'}$.   
\par 
Part $(2)$ follows from part $(1)$ by putting $x=0$ since $K_\alpha$ is radial. 
\end{proof}

Let $f\in L^p_w$ and $E_\alpha(f) \in L^p_w$.  We take $\varphi\in 
\mathscr{S}(\Bbb R^n)$ satisfying $\int \varphi(x)\, dx=1$. 
Let 
$f^{(\epsilon)}(x)=\varphi_\epsilon*f(x)$ and $g^{(\epsilon)}(x)=
J_{-\alpha}(\varphi_\epsilon)*f(x)$.  Then, note that $g^{(\epsilon)}
\in L^p_w$ and $f^{(\epsilon)}=J_\alpha(g^{(\epsilon)})$ by part (1) of 
Lemma 5.6.  
\par 
By \eqref{4.1} we have 
\begin{equation}\label{4.8} 
\|E_\alpha(f^{(\epsilon)})\|_{p,w}+\|f^{(\epsilon)}\|_{p,w}\simeq 
\|g^{(\epsilon)}\|_{p,w}.  
\end{equation}  
We note that  
\begin{equation}\label{4.9} 
\sup_{\epsilon>o}
\|f^{(\epsilon)}\|_{p,w}\leq C\|M(f)\|_{p,w}\leq C\|f\|_{p,w}. 
\end{equation}
Also,  Minkowski's inequality implies that 
\begin{align*}
E_\alpha(f^{(\epsilon)})(x)&=
\left(\sum_{k=-\infty}^\infty\left|\varphi_\epsilon*f(x)-
\Phi_{2^k}*\varphi_\epsilon*f(x)\right|^2 2^{-2k\alpha} \right)^{1/2} 
\\ 
&\leq \int_{\Bbb R^n}|\varphi_\epsilon(y)|
\left(\sum_{k=-\infty}^\infty\left|f(x-y)-
\Phi_{2^k}*f(x-y)\right|^2 2^{-2k\alpha} \right)^{1/2}\, dy 
\\ 
&\leq CM(E_\alpha(f))(x).  
\end{align*}   
Thus 
$$\sup_{\epsilon>0}\|E_\alpha(f^{(\epsilon)})\|_{p,w} 
\leq C\|M(E_\alpha(f))\|_{p,w}
\leq C\|E_\alpha(f)\|_{p,w}, $$ 
which combined with \eqref{4.8} and \eqref{4.9} implies that 
$\sup_{\epsilon>0}\|g^{(\epsilon)}\|_{p,w}<\infty$.   
\par 
Therefore we can choose 
a sequence $\{g^{(\epsilon_k)}\}$, $\epsilon_k \to 0$, 
 which converges weakly in $L^p_w$.  
Let $g^{(\epsilon_k)}\to g$ weakly in $L^p_w$. 
Then, since $\{f^{(\epsilon_k)}\}$ converges to $f$ in $L^p_w$, 
we can conclude that $f=J_\alpha(g)$. To show this, let 
$\Lambda_h(f)=\int f(x)h(x)\,dx$ for $h\in \mathscr{S}(\Bbb R^n)$. 
Then it is easy to see that $\Lambda_h$ is a bounded linear functional 
on $L^p_w$ for every $h\in \mathscr{S}(\Bbb R^n)$.  Thus, for any 
$h\in \mathscr{S}(\Bbb R^n)$, applying part (2) of Lemma 5.6  
 and noting $J_\alpha(h)\in \mathscr{S}(\Bbb R^n)$,  we have 
\begin{align*} 
\int f(x)h(x)\, dx&=\lim_k\int f^{(\epsilon_k)}(x)h(x) \, dx
=\lim_k \int J_\alpha(g^{(\epsilon_k)})(x)h(x)\, dx 
\\ 
&=\lim_k \int g^{(\epsilon_k)}(x)J_\alpha(h)(x)\, dx 
  =\int g(x)J_\alpha(h)(x)\, dx 
\\ 
&=\int J_\alpha(g)(x)h(x)\, dx. 
\end{align*}   
This implies that $f=J_\alpha(g)$ and hence $f\in W^{\alpha, p}_w(\Bbb R^n)$.  
This completes the proof of Corollary 5.4. 
\end{proof}

\begin{remark}\label{one} 
Let $\psi=\sgn -\sgn*\Phi$ on $\Bbb R$, where $\Phi\in \mathscr M^1$.   
We note that 
$\hat{\psi}(\xi)=-i\pi^{-1}\xi^{-1}(1-\hat{\Phi}(\xi))$. 
We have results analogous to Theorems $5.1$ and $5.3$ for $g_\psi$ and 
$\Delta_\psi$, respectively, with similar proofs. They can be applied to prove 
results generalizing Corollaries $5.2$ and $5.4$ to the case $n=1$ and 
$\alpha=1$ by arguments similar to those used for the corollaries. 
\end{remark}

\section{Proof of \eqref{mar}} 
In this section we give a proof of the estimate \eqref{mar} for completeness. 
Put $\psi= \varphi^{(\alpha)}$.  
To prove \eqref{mar}, assuming $|y|<|x|/2$,  we write 
\begin{equation*}
L=\int_0^\infty|t^{-1}\psi((x-y)/t)-t^{-1}\psi(x/t)|^2\, \frac{dt}{t}.  
\end{equation*}  
We first assume $x>0$ and $y>0$.  By the change of variables $x/t=u$ we have 
\begin{equation*}  
L=x^{-2}\int_0^\infty|\psi(u-uy/x)-\psi(u)|^2 u\, du =I+II, 
\end{equation*}  
where 
\begin{gather*} 
 I=x^{-2}\int_0^1 |\psi(u-uy/x)-\psi(u)|^2 u\, du, 
 \\ 
II= x^{-2}\int_1^\infty|\psi(u-uy/x)-\psi(u)|^2 u\, du. 
\end{gather*} 
\par 
We estimate $I$ and $II$ separately. We see that 
\begin{equation*}  
II=x^{-2}\int_1^\infty|\psi(u-uy/x)|^2 u\, du 
=\alpha^2 x^{-2}\int_1^{x/(x-y)} (1-|u(1-y/x)|)^{2(\alpha-1)}u\, du.  
\end{equation*}  
Thus, by the change of variables $w=u(x-y)/x$, we have 
\begin{equation*}  
II=\alpha^2(x-y)^{-2}\int_{(x-y)/x}^1 (1-w)^{2(\alpha-1)}w\, dw
\leq \alpha^2(x-y)^{-2}\int_{(x-y)/x}^1 (1-w)^{2(\alpha-1)}\, dw,  
\end{equation*}  
which implies that 
\begin{equation}\label{e6}   
II\leq  \alpha^2(x-y)^{-2}(2\alpha-1)^{-1}(y/x)^{2\alpha-1}\leq  
C_\alpha y^{2\alpha-1}x^{-1-2\alpha}.   
\end{equation}  
To deal with $I$, we write 
\begin{equation*}  
I=\alpha^2 x^{-2}\int_{0}^1|(1-u(1-y/x))^{\alpha-1}-(1-u)^{\alpha-1}|^2u\, du
=I_1+I_2,  
\end{equation*}  
where 
\begin{gather*} 
 I_1=\alpha^2x^{-2}\int_0^{1-2y/x} |(1-u(1-y/x))^{\alpha-1}-(1-u)^{\alpha-1}|^2u\, du, 
 \\ 
I_2= \alpha^2x^{-2}\int_{1-2y/x}^1 |(1-u(1-y/x))^{\alpha-1}-(1-u)^{\alpha-1}|^2u\, du. 
\end{gather*}
We observe that 
\begin{align}\label{e7} 
&\int_{1-2y/x}^1 (1-u(1-y/x))^{2(\alpha-1)}\, du 
=x(x-y)^{-1}\int_{(1-2y/x)(x-y)/x}^{(x-y)/x}(1-w)^{2(\alpha-1)}\, dw 
\\ 
&\leq C_\alpha\left(\left(1- (1-2y/x)(x-y)/x\right)^{2\alpha-1} 
-\left(1- (x-y)/x\right)^{2\alpha-1}\right)        \notag 
\\ 
&\leq C_\alpha (y/x)^{2\alpha-1}.                  \notag
\end{align}
Also, we have 
\begin{equation}\label{e8} 
\int_{1-2y/x}^1 (1-u)^{2(\alpha-1)}\, du 
\leq C_\alpha (y/x)^{2\alpha-1}.
\end{equation}
By \eqref{e7} and \eqref{e8} we see that 
\begin{equation}\label{e91} 
I_2\leq C_\alpha x^{-2}(y/x)^{2\alpha-1}=C_\alpha y^{2\alpha-1}x^{-1-2\alpha}. 
\end{equation}
To estimate $I_1$ we recall that $1/2<\alpha < 3/2$.  By the mean value 
theorem, we have 
\begin{align}\label{e9} 
I_1&\leq C x^{-2} (y/x)^2 \int_0^{1-2y/x} (1-u)^{2(\alpha-2)}\, du  
\\ 
&\leq Cx^{-2} (y/x)^2 (2y/x)^{2\alpha-3}=Cy^{2\alpha-1}x^{-2\alpha-1}.  
\notag 
\end{align}
The estimate $I\leq C_\alpha y^{2\alpha-1}x^{-1-2\alpha}$ follows from 
\eqref{e91} and \eqref{e9}, which combined with \eqref{e6} implies 
\begin{equation*}\label{e10} 
L\leq C_\alpha y^{2\alpha-1}x^{-1-2\alpha},  
\end{equation*}
when $x>0$, $y>0$.  
\par 
Next we deal with the case $x>0$, $y<0$. In this case we also consider 
the analogous decomposition $L=I+II$. 
 Since $\psi$ is supported in $[-1, 1]$ 
and $x>0, y<0$, we see that $II=0$. 
Also, $I=I_1+I_2$, where   
\begin{align*} 
 I_1&=\alpha^2x^{-2}\int_0^{1-2|y|/x} |(1-u(1-y/x))^{\alpha-1}-(1-u)^{\alpha-1}|^2u
 \, du, 
 \\ 
I_2&= x^{-2}\int_{1-2|y|/x}^1 |\psi(u(1-y/x))-\psi(u)|^2 u\, du. 
\end{align*}
To estimate $I_2$, we see that 
\begin{align*}\label{e11}  
\int_{1-2|y|/x}^1 |\psi(u(1-y/x))|^2 u\, du 
&\leq \alpha^2\int_{1-2|y|/x}^{x/(x-y)} |1-u(1-y/x)|^{2(\alpha-1)} \, du  
\\ 
&=\alpha^2x(x-y)^{-1}\int_{(x-y)(x+2y)/x^2}^1  (1-w)^{2(\alpha-1)}\, dw   
\notag 
\\ 
&=\alpha^2x(x-y)^{-1}(2\alpha-1)^{-1}(|y|/x+ 2(y/x)^2)^{2\alpha-1}       
\notag 
\\ 
&\leq C_\alpha |y/x|^{2\alpha-1}.                            \notag 
\end{align*}  
Similarly, 
\begin{equation*}
\int_{1-2|y|/x}^1 (1-u)^{2(\alpha-1)} u\, du \leq C_\alpha  |y/x|^{2\alpha-1}. 
\end{equation*}
Thus 
\begin{equation}\label{e12}  
I_2\leq C  |y|^{2\alpha-1}x^{-2\alpha-1}. 
\end{equation}  
On the other hand, by the mean value theorem, 
\begin{align}\label{e13}  
I_1
&\leq \alpha^2 x^{-2}\int_0^{1-2|y|/x} 
\left(|y|x^{-1}|\alpha-1||(1-u(1-y/x))|^{\alpha-2}  \right)^2\, du   
\\ 
&\leq C y^2x^{-4} x(x-y)^{-1}\int_0^{(x+2y)(x-y)/x^2}(1-u)^{2(\alpha-2)}\, du 
\notag 
\\ 
&=Cy^2x^{-3}(x-y)^{-1}(3-2\alpha)^{-1}\left((|y|/x+ 2(y/x)^2)^{2\alpha-3}-1
\right)    \notag 
\\ 
&\leq C_\alpha |y|^{2\alpha-1}x^{-2\alpha-1}.                 \notag
\end{align}  
The estimates \eqref{e12} and \eqref{e13} imply that $I\leq 
C |y|^{2\alpha-1}x^{-2\alpha-1}$ for $x>0, y<0$.  
\par 
Since $\psi$ is odd, we observe that 
\begin{equation*}  
L=\int_0^\infty|t^{-1}\psi((-x+y)/t)-t^{-1}\psi(-x/t)|^2\, \frac{dt}{t}.  
\end{equation*} 
Thus, the results for the cases $x<0$, $y>0$ and $x<0$, $y<0$ will follow 
from the results for the cases $x>0$, $y<0$ and $x>0$, $y>0$, respectively.


\begin{thebibliography}{99} 

\bibitem{AMV} 
R. Alabern, J. Mateu and J. Verdera, {\it A new characterization of Sobolev 
spaces on $\Bbb R^n$},  Math. Ann. {\bf 354} (2012), 589--626.  

\bibitem{BCP} A. Benedek, A.~P.~Calder\'on and R. Panzone,   
{\it Convolution operators on Banach space valued functions},     
 Proc.\ Nat.\ Acad.\ Sci.\ U.~S.~A.~ {\bf 48}  (1962),  356--365.  

\bibitem{BL}
J. Bergh and J. L\"{o}fstr\"{o}m, {\it Interpolation Spaces. An Introduction}, 
Grundlehren der mathematischen Wissenschaften 223. Berlin-Heidelberg-New York, 
Springer-Verlag, 1976. 

  \bibitem{CF}  R. R. Coifman and C. Fefferman, 
{\it Weighted norm inequalities for maximal functions and singular 
integrals}, Studia Math. {\bf 51} (1974), 241--250.  


\bibitem{CZ} A. P. Calderon and A. Zygmund, 
{\it Algebras of certain singular operators}, 
Amer. J.  Math.  {\bf 78} (1956), 310--320.  

\bibitem{D}
 J.~Duoandikoetxea, 
{\it Sharp $L^p$ boundedness for a class of square functions},
 Rev Mat Complut {\bf 26} (2013), 535-548.

\bibitem{DR} 
 J.~Duoandikoetxea and J.~L.~ Rubio de Francia, 
 {\it Maximal and singular integral operators via Fourier transform 
 estimates},   Invent. Math. {\bf 84} (1986), 541--561. 

 
\bibitem{FS}  D. Fan and S. Sato,  
{\it Remarks on Littlewood-Paley functions and singular 
integrals}, 
 J. Math. Soc. Japan {\bf 54} (2002),  565--585.  
\bibitem{FeS}
C. Fefferman and E. M. Stein,    
{\it $H^p$  spaces of several variables}, 
Acta Math. {\bf 129} (1972),  137--193.  
\bibitem{GR} J. Garcia-Cuerva and  J.L. Rubio de Francia,  
{\it Weighted Norm Inequalities and Related Topics}, 
North-Holland, Amsterdam, New York, Oxford, 1985.  

\bibitem{HL} 
 P. Haj{\l}asz, Z. Liu, {\it A Marcinkiewicz integral type characterization of 
 the Sobolev space},  arXiv:1405.6127 [math.FA].  


\bibitem{H}
 L.~H\"ormander,  
{\it Estimates for translation invariant operators in $L^p$ spaces},  
 Acta Math.\  {\bf 104}  (1960),    93--139.   
 
\bibitem{KS}  M.~Kaneko and G.~Sunouchi,  
{\it On the Littlewood-Paley and Marcinkiewicz functions in higher 
dimensions},    
 T\^ohoku Math. J. {\bf 37}  (1985),  343--365.  

 \bibitem{KW} D. S. Kurz and R. L. Wheeden, 
 {\it Results on weighted norm inequalities for multipliers}, Trans. Amer. 
 math. Soc. {\bf 255} (1979), 343--362. 

\bibitem{MW}  B. Muckenhoupt and R. L. Wheeden,   
{\it Norm inequalities for the Littlewood-Paley function $g_\lambda^*$},  
 Trans. Amer. Math. Soc.
 {\bf 191} (1974),  95--111.  

 
\bibitem{Ru} J.L. Rubio de Francia, 
{Factorization theory and $A_p$ weights},  
Amer. J. Math.  {\bf 106} (1984),  533--547.   

\bibitem{RRT} 
 J. L. Rubio de Francia, F. J. Ruiz and J. L. Torrea,  
{\it Calder\'on-Zygmund theory for operator-valued kernels},  
 Adv. in Math. {\bf 62} (1986), 7--48.  

\bibitem{Sa} 
 S. Sato,    
{\it Remarks on square functions in the Littlewood-Paley theory},  
 Bull.\ Austral.\ Math.\ Soc.\  {\bf 58}  (1998),   199--211.  

 \bibitem{Sa2} 
 S.~Sato, 
{\it Multiparameter Marcinkiewicz integrals and a resonance
theorem},
Bull.\ Fac.\ Ed.\ Kanazawa Univ.\ Natur.\ Sci.\ {\bf48} (1999), 1--21. 
(http://hdl.handle.net/2297/25017)   
 \bibitem{Sa3} S. Sato,
 {\it Estimates for Littlewood-Paley functions and extrapolation}, 
Integr. equ. oper. theory {\bf 62} (2008), 429--440.  

 \bibitem{Sa4} S. Sato, {\it Littlewood-Paley operators and Sobolev spaces}, 
 Illinois J. Math. {\bf 58} (2014), 1025-1039.

\bibitem{St}
  E.~M.~Stein,    
{\it On the functions of Littlewood-Paley, Lusin, and Marcinkiewicz}, 
 Trans.\ Amer.\ Math.\ Soc.\   {\bf 88}  (1958),  430--466.   
 \bibitem{St2}
  E.~M.~Stein, {\it The characterization of functions arising as potentials},  
Bull. Amer. Math. Soc.  {\bf 67} (1961),  102--104.   

 \bibitem{St3}
  E.~M.~Stein,  
  {\it Singular Integrals and Differentiability Properties of Functions}, 
  Princeton Univ.\ Press, 1970.  
  \bibitem{Str} 
R. S.  Strichartz, {\it  Multipliers on fractional Sobolev spaces}, 
J. Math. Mech. {\bf 16} (1967), 1031--1060. 
  
\bibitem{ST}
  J.~-O.~Str\"{o}mberg and A.~Torchinsky, 
  Weighted Hardy spaces,   
  Lecture Notes in Math.  1381,  Springer-Verlag,  
 Berlin Heidelberg New York London Paris Tokyo  Hong Kong, 1989.    

\bibitem{Su} 
 G. Sunouchi,  
{\it On the functions of Littlewood-Paley and Marcinkiewicz},   
T\^ohoku Math. J. {\bf 36} (1984),  505--519.   
\bibitem{U} A. Uchiyama, {\it Characterization of $H^p(\Bbb R^n)$ in terms of 
generalized Littlewood-Paley $g$-functions}, Studia Math. {\bf 81} (1985), 
135--158.   

\bibitem{Wh}  R. L.  Wheeden, {\it Lebesgue and Lipschitz spaces and 
integrals of the Marcinkiewicz type}, Studia Math. {\bf 32} (1969), 73--93.  

 
  \bibitem{Z} A.~Zygmund, 
{\it Trigonometric Series},  
 2nd ed.,   Cambridge Univ. Press,  
 Cambridge, London, New York and Melbourne, 1977.  

\end{thebibliography}
\end{document}